\documentclass[10pt,a4paper]{article}
\usepackage[utf8]{inputenc}
\usepackage[english]{babel}
\usepackage{amsmath,amssymb,amsfonts,amsthm,mathrsfs}
\usepackage{yfonts}
\usepackage{graphicx,color}
\usepackage{epsfig}
\usepackage{bigints}

\usepackage{bbm}
\usepackage{indentfirst}
\usepackage{mathtools}
\usepackage{hyperref}

\usepackage{enumitem}
\usepackage{caption}
\usepackage{subcaption}
\usepackage{pb-diagram}
\usepackage[all]{xy}
\usepackage{ulem}

\newtheorem{theorem}{Theorem}[section]

\newtheorem{lemma}{Lemma}[section]
\newtheorem{proposition}{Proposition}[section]
\newtheorem{remark}{Remark}[section]

\def\mc{\mathcal}
\def\mf{\mathfrak}
\def\ms{\mathscr}

\def\a{\alpha}

\def\t{\theta}
\def\d{\delta}

\def\g{\gamma}

\def\l{\lambda}
\def\p{\partial}

\def\e{\varepsilon}
\def\v{\varphi}

\def\k{\kappa}



\title{\vskip-2.5cm
	{High multiplicity and global structure of coexistence states in a predator-prey model with saturation}
	\thanks{This paper has been written under the auspices of the Ministry of Science and Innnovation of Spain under Research Grants PID2021-123343NB-I00 and PID2024-155890NB-I00, and  the Institute of Interdisciplinary Mathematics of Complutense University of Madrid. K. Kuto has been partially supported by JSPS KAKENHI Grant Number JP25K00917.} }

\author{
	\sc Kousuke Kuto
	\\
	\small Waseda University
	\\
	\small Department of Applied Mathematics
	\\
	\small 3-4-1 Okubo, Shinjuku-ku, Tokyo 169-8555, Japan
	\\
	\small E-mail: {\tt kuto@waseda.jp }
	\medskip
	\\
	\sc Juli\'an L\'opez-G\' omez and Eduardo Mu\~{n}oz-Hern\'andez
	\\
	\small Universidad Complutense de Madrid
	\\
	\small Instituto de Matem\'{a}tica Interdisciplinar (IMI)
	\\
	\small Departamento de An\'alisis Matem\'atico y Matem\'atica
	Aplicada
	\\
	\small  Plaza de las Ciencias 3, 28040   Madrid, Spain
	\\
	\small E-mails: {\tt   julian@mat.ucm.es} and {\tt eduardmu@ucm.es}
	\bigskip
}
\begin{document}
	\maketitle
\begin{abstract}
This paper  establishes that, under the appropriate range of values of the parameters involved in 
the formulation of the model, a diffusive predator-prey system with saturation can have an arbitrarily large number of coexistence states for sufficiently large saturation rates. 
Moreover, it ascertains the global structure of the set of coexistence states in the limiting system as the saturation rate blows-up.

\vspace{1cm}
\end{abstract}

\section{Introduction}\label{sec1}

\noindent This paper studies the existence, multiplicity and global structure of the coexistence states of the following diffusive one-dimensional Holling--Tanner predator-prey model with Neumann boundary conditions
\begin{equation}
\label{1.i}
\left\{
\begin{array}{ll}
-u''=\lambda u - a(x)u^2 - b\frac{uv}{1+ \g u}\qquad \hbox{in}\;\;(0,1),\\[6pt]
-v''=\mu v -dv^2+ c(x)\frac{uv}{1+ \g u}\qquad\,\hbox{in}\;\;(0,1),\\[7pt]
u'(0)=u'(1)=0,\;\; v'(0)=v'(1)=0,
\end{array}
\right.	
\end{equation}
where $a,c\in\mc{C}([0,1];\mathbb{R})\setminus\{0\}$ are non-negative, $b>0$ and $d>0$ are positive constants, and $\l$, $\mu$ and $\g>0$ are viewed as real parameters. This model is a special one-dimensional prototype  of the generalized heterogeneous  diffusive Holling--Tanner model introduced by two of the authors in  \cite{LGMH20} and \cite{LGMH24}, whose simplest prototypes go back to Blat and Brown \cite{BB86} and Casal et al. \cite{CEL}, where some necessary and sufficient conditions for the existence of coexistence states was established in its multidimensional counterpart under Dirichlet boundary conditions. The limiting problem when $\g=0$ provides us with a classical diffusive predator-prey model of Lotka--Volterra type. In Population Dynamics, $\g>0$ represents the saturation rate of the predator population, whose density is measured by $v$, in the presence of prey abundance, whose density is measured by $u$. The coexistence states are the componentwise positive solutions, $(u,v)$, of \eqref{1.i}. According to the strong maximum principle, they satisfy $u\gg 0$ and $v\gg 0$ in the sense that $u(x)>0$ and $v(x)>0$ for all $x\in [0,1]$ 
(see Chapters 1 and 2 of \cite{LG13}, if necessary). 
\par
\par
Whereas in the special case when $\g=0$ the set of values of the parameters $(\l,\mu)$ where \eqref{1.i} has a coexistence state coincides with the region where both semitrivial positive solutions are linearly unstable, when $\g>0$, \eqref{1.i} can have coexistence states even in the region where  one of the semitrivial positive solutions becomes linearly stable. Actually, in this region one might have multiple coexistence states, as documented numerically by Casal et al. \cite{CEL}. The rigorous existence of an $S$-shaped bifurcation diagram of coexistence states was established later by Du and Lou \cite{DL98} 
and two of the authors in \cite{LGMH24}. A similar $S$-shaped bifurcation diagram has also been obtained in a quasilinear predator-prey steady-state model with cross-diffusion,
which, after a suitable change of variables, is reduced to a system resembling the Holling–Tanner type, as shown by Kuto and Yamada \cite{KutoYamada2004} (see also \cite{Kuto2004}).

\par
The multiplicity results in the Holling--Tanner model \eqref{1.i} contrast dramatically with the
one-dimensional uniqueness theorem of L\'opez--G\'omez and Pardo  \cite{LP93}, \cite{LP98}. Based on this uniqueness  theorem, it was shown in \cite{CEL} and \cite{LGMH20} that \eqref{1.i} has a unique coexistence state for sufficiently small $\g>0$.
\par
The main goal of this paper is to show that \eqref{1.i} can have an arbitrarily large number of coexistence states for sufficiently large $\g>0$. This result sharpens,  substantially, a previous finding of Du and Lou \cite{DL01} for a multidimensional counterpart of \eqref{1.i}, where the existence of two non-constant positive coexistence states was established via variational methods. 
\par
Precisely,  by introducing the change of variables
\begin{equation}
\label{1.ii}
w:=\g \,u,\qquad \varepsilon=\frac{1}{\g},
\end{equation}
\eqref{1.i} can be equivalently expressed as
\begin{equation}
\label{1.iii}
\left\{
\begin{array}{ll}
-w''=\lambda w - \varepsilon a(x)w^2 - b\frac{wv}{1+w}\qquad \hbox{in}\;\;(0,1),\\[6pt]
-v''=\mu v -dv^2+ \varepsilon c(x)\frac{wv}{1+w}\qquad\,\,\,\,\hbox{in}\;\;(0,1),\\[7pt]
w'(0)=w'(1)=0,\;\; v'(0)=v'(1)=0.
\end{array}
\right.	
\end{equation}
Thanks to \eqref{1.ii}, analyzing \eqref{1.i} as $\g\uparrow+\infty$ is equivalent to study \eqref{1.iii} as $\e\downarrow0$. Moreover, it is natural to think of \eqref{1.iii} as a perturbation of the underlying limiting system obtained from \eqref{1.iii} by switching off to zero the parameter $\e$ as follows
$$
\left\{
\begin{array}{ll}
-w''=\lambda w- b\frac{wv}{1+w}\qquad \hbox{in}\;\;(0,1),\\[6pt]
-v''=\mu v -dv^2\qquad\quad\,\hbox{in}\;\;(0,1),\\[7pt]
w'(0)=w'(1)=0,\;\; v'(0)=v'(1)=0.
\end{array}
\right.	
$$
Thus, as this limiting system is uncoupled, and $v= \frac{\mu}{d}$ is the unique positive solution of
the $v$-equation, which exists if and only if $\mu>0$, this paper will focus special attention into the positive solutions of the problem
\begin{equation}
\label{1.iv}
\left\{
\begin{array}{ll}
-w''=\lambda w- \frac{b\mu}{d}\frac{w}{1+w}\qquad \hbox{in}\;\;(0,1),\\[6pt]
w'(0)=w'(1)=0,
\end{array}
\right.	
\end{equation}
with $\mu>0$ and $\l>0$, since $w=0$ is the unique non-negative solution of \eqref{1.iv} if
$\l\leq0$.  Indeed, suppose that \eqref{1.iv} admits a positive solution, $w$. Then, $w(x)>0$
for all $x\in [0,1]$. Moreover, there exists $x_m\in [0,1]$ such that
$$
   w(x_m)= \max_{x\in [0,1]}w(x).
$$
Necessarily, $w''(x_m)\leq 0$. Thus,
$$
  \l - \frac{b\mu}{d}\frac{1}{1+w(x_m)}\geq 0,
$$
which implies $\l>0$. Therefore, throughout this paper we will assume that
\begin{equation}
\label{1.5}
   \l>0\quad\hbox{and}\quad \mu>0.
\end{equation}
Since \eqref{1.iv} is an autonomous second order problem, phase plane analysis and time-map techniques can be used to get the next result, which is our main multiplicity theorem for \eqref{1.iv}. Throughout this paper,
we set
\begin{equation*}
 \mu_\k:=\frac{d}{b}(2\k\pi)^2,\qquad\k\in\mathbb{N}\cup\{0\}.
\end{equation*}

\begin{theorem}
\label{th1.1}
Suppose  $\mu\in (0,\mu_1]$. Then, the constant
$$
   w_0:=\frac{b\mu}{d\l}-1
$$
is the unique positive solution of \eqref{1.iv}.
\par
Suppose  $\mu\in (\mu_\k,\mu_{\k+1}]$ for some integer $\k \geq 1$. Then,
the set of positive solutions of
\eqref{1.iv} consists of the constant solution $w_0$ plus $\k$ closed nested loops, $\mathscr{C}_n$,
$n\in\{1,...,\k\}$, bifurcating from $w_0$ at a family of critical values $\l_{n}^\pm\equiv \l_{n}^\pm(\mu)$ such that
\begin{equation*}
0<\l_{1}^-<\l_{2}^-<\cdots<\l_{\k}^-<\tfrac{b\mu}{2d}<\l_{\k}^+<\cdots<\l_{2}^+
<\l_{1}^+<\tfrac{b\mu }{d},
\end{equation*}
as illustrated in Figure \ref{Fig4} for $\k=2$. Moreover, for every
$n\in\{1,...,\k\}$,
$$
  \mathcal{P}_\l (\mathscr{C}_n)=[\l_{n}^-,\l_{n}^+],
$$
where $\mc{P}_\l$ denotes the $\l$-projection operator,
$$
   \mc{P}_\l(\l,w)=\l,
$$
and, for every  $\l \in (\l_{n}^-,\l_{n}^+)$, 
$\mathscr{C}_n$ consists of two positive solutions, $(\lambda, w_n)$ and $(\lambda, \tilde w_n)$, where $w_n$ and $\tilde w_n$ are the only positive solutions of \eqref{1.iv}, $w$, such that $w-w_0$ has $n$ zeros in $(0,1)$. 
Therefore, \eqref{1.iv} has exactly $2\k+1$ positive solutions for each $\l\in (\l_{\k}^-,\l_{\k}^+)$.
\par
Furthermore, 
the loop $\mathscr{C}_\k$ bifurcates from $\mathscr{C}_0$ at the single point
$$
(\mu,\l,w_0)=\big(\mu_\k,\tfrac{b\mu}{2d},1\big)
$$
as $\mu$ crosses $\mu_\k$, and its existence interval, $\l\in (\l_{\k}^-,\l_{\k}^+)$, is increasing with respect to $\mu$.
\end{theorem}

\noindent  As it will become apparent in Section 2, $\l_{\k}^{\pm}$ are the real roots of the eigencurves of the linearized problem at the constant solution $w_0$.
\par
The second aim of this paper is to show that the high multiplicity result established by Theorem \ref{th1.1}
is mimicked by the perturbed problem \eqref{1.iii} for sufficiently small $\e>0$, except, at most, at finitely many values of the parameter $\l$.   This is an extremely delicate issue based on the non-degeneration of
the coexistence states of \eqref{1.iii} for $\e=0$, which relies on the structure of the solution
set of \eqref{1.iv} and the analyticity of the eigenvalues of the associated variational problems. As a consequence of their non-degeneration, the Implicit Function Theorem shows that \eqref{1.iii} has an arbitrarily large number of coexistence states for sufficiently small $\e>0$ provided $\mu$ is sufficiently large, which is the main finding of this paper.
\par
This paper is organized as follows. Section 2 delivers the proof of Theorem \ref{th1.1}. Section 3 first ascertains the dimension of the unstable manifolds of the solutions of \eqref{1.iii} for $\l$ sufficiently
close to $\l_n^\pm$ through the Exchange Stability Principle of Crandall and Rabinowitz \cite{CR73}, and then combines the underlying analyticity with respect to $\l$ together with the exact multiplicity result established by Theorem \ref{th1.1} to get the dimensions of the unstable manifolds of the positive solutions of \eqref{1.iv} for all value of $\l$ outside a certain finite set. Finally, in Section 4 the multiplicity theorem for the problem \eqref{1.iii} is stated and proven.

\setcounter{equation}{0}
\section{The limiting system}
\label{sec2}

\noindent In this section, we analyze the constant and nonconstant coexistence states of the limiting system derived in Section \ref{sec1}
\begin{equation}
\label{2.i}
\left\{
\begin{array}{ll}
-w''=\lambda w- b\frac{wv}{1+w}\qquad \hbox{in}\;\;(0,1),\\[4pt]
-v''=\mu v -dv^2\qquad\quad\,\hbox{in}\;\;(0,1),\\[4pt]
w'(0)=w'(1)=0,\;\; v'(0)=v'(1)=0.
\end{array}
\right.	
\end{equation}
As \eqref{2.i} is uncoupled and the unique positive solution of
\begin{equation*}
\left\{
\begin{array}{l}
-v''=\mu v -dv^2\qquad\,\hbox{in}\;\;(0,1),\\[2pt]
v'(0)=v'(1)=0,
\end{array}\right.	
\end{equation*}
is the constant function
$$
v_0:=\frac{\mu}{d},
$$
the problem \eqref{2.i} reduces to the analysis of
\begin{equation}
\label{2.2}
\left\{
\begin{array}{ll}
-w''=\lambda w- \frac{b\mu}{d}\frac{w}{1+w}\qquad \hbox{in}\;\;(0,1),\\[4pt]
w'(0)=w'(1)=0,
\end{array}
\right.
\end{equation}
which coincides with \eqref{1.iv}. Thus, throughout this section we will deal with \eqref{2.2}, whose unique constant positive solution is given by
$$
w_0:=\tfrac{b\mu}{d\l}-1,
$$
for all $\mu>0$ and $\l\in(0,\tfrac{b\mu}{d})$.
Regarding $\l$ as the main bifurcation parameter, Figure \ref{Fig1} plots the curve of constant positive solutions of \eqref{2.2},
$$
\mathscr{C}_0:=\left\{(\l,w_0)\;:\;\l\in (0,\tfrac{b\mu}{d})\right\}.
$$
It should be noted that, for every $\mu>0$,  $\l\mapsto w_0$ is decreasing with respect to $\l$ and
$$
\lim_{\l\downarrow 0}w_0=+\infty,\quad \lim_{\l\uparrow \tfrac{b\mu}{d}}w_0=0.
$$
\begin{figure}[h!]
\centering
\includegraphics[scale=0.38]{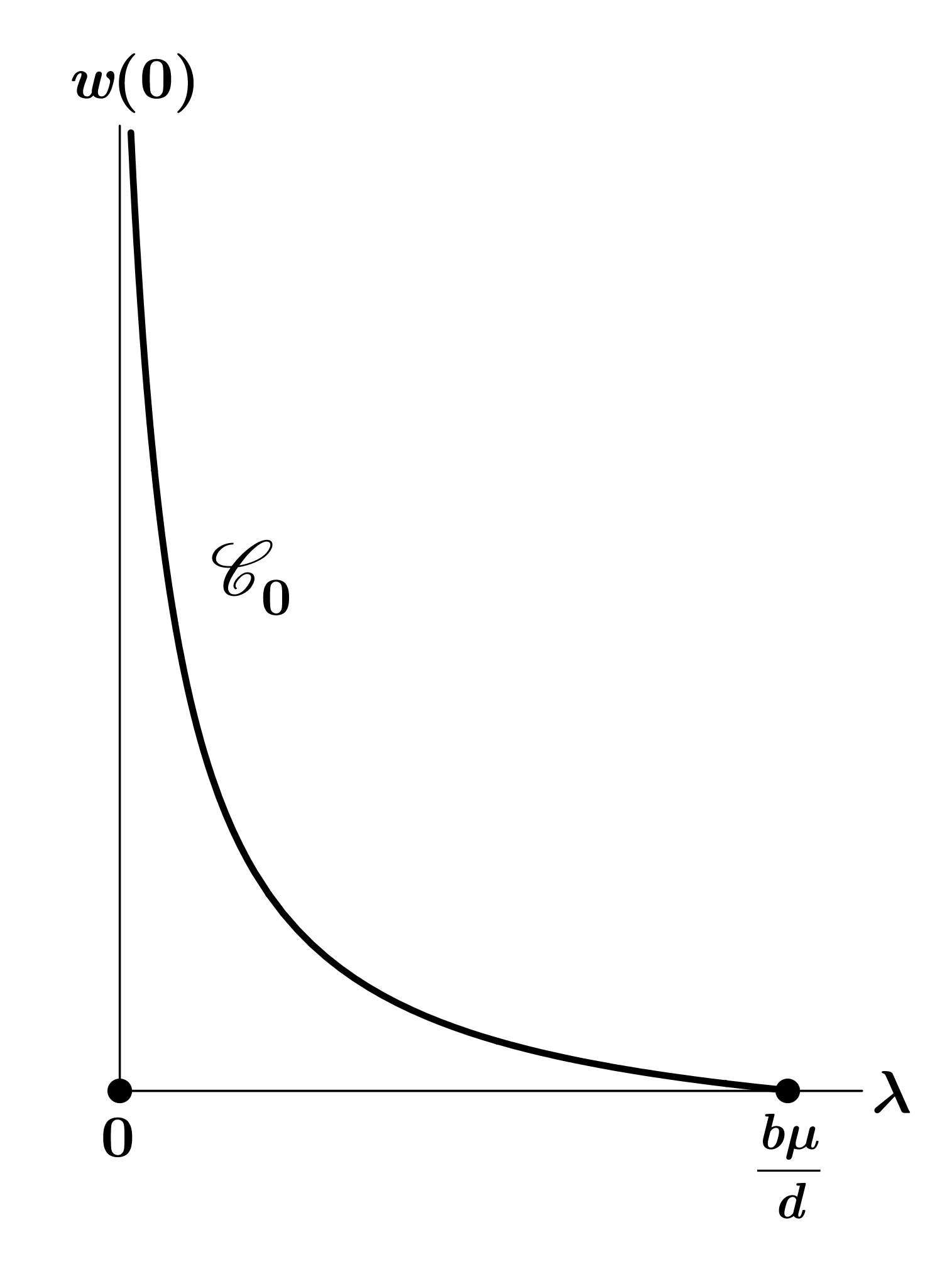}
\caption{The curve $\mathscr{C}_0$ for any fixed $\mu>0$.}
\label{Fig1}
\end{figure}

Actually, as it will become apparent in Lemma \ref{le2.1}, the condition $\l\in (0,\tfrac{b\mu}{d})$ is necessary for the existence of a positive solution (constant or not) of the problem \eqref{2.2}.

\subsection{Spectrum of the linearization around $\mathscr{C}_0$}

\noindent This section analyzes the spectrum of the linearized problem of \eqref{2.2} at the constant positive solution $w_0$. This local analysis will provide us with some pivotal necessary and sufficient conditions for the bifurcation of nonconstant positive solutions from $\mathscr{C}_0$. Subsequently, $\l$ will be the main bifurcation parameter of \eqref{2.2}, while $\mu$ will play the role of a secondary parameter. The bigger is $\mu$, the more complex is the global bifurcation diagram in $\l$ of the problem \eqref{2.2}.
\par
Setting $D^2=\frac{\partial^2}{\p x^2}$, it is folklore that, for every $V\in L^{\infty}(0,1)$, the linear
eigenvalue problem
\begin{equation*}
\left\{
\begin{array}{ll}
(-D^2+V(x))\varphi=\tau\varphi\quad \hbox{in}\;\;(0,1),\\[7pt]
\varphi'(0)=\varphi'(1)=0,
\end{array}
\right.
\end{equation*}
possesses an increasing sequence of simple eigenvalues, which will be denoted by
$$
\{\tau_\ell(V)\}_{\ell=0}^{\infty},
$$
whose associated eigenfunctions have exactly $\ell$ zeros in $(0,1)$. In particular, in the special case when $V(x)\equiv V$ is constant, this sequence is given by
\begin{equation}
\label{2.3}
     \tau_\ell(V):=\tau_\ell(0)+V=(\ell\pi)^2+V,\qquad \ell\in\mathbb{N}\cup\{0\}.
\end{equation}
The next result provides us with a necessary condition for the existence of a positive solution of
\eqref{2.2}. Note that, throughout this paper we are imposing \eqref{1.5}.

\begin{lemma}
\label{le2.1}
Suppose that \eqref{2.2} admits a positive solution. Then, $\l\in (0,\tfrac{b\mu}{d})$.
\end{lemma}
\begin{proof}
Let $w$ be a positive solution of \eqref{2.2}. Then,
$$
   \left(-D^2+\tfrac{b\mu}{d}\tfrac{1}{1+w}\right)w =\l w \quad \hbox{in}\;\ (0,1).
$$
Thus, since $w(x)>0$ for each $x\in [0,1]$, we find that
$$
\l=\tau_0\left(\tfrac{b\mu}{d}\tfrac{1}{1+w}\right).
$$
Hence, by the monotonicity of the principal eigenvalue with respect to the potential (see \cite[Sec. 8.1]{LG13}, if necessary), it follows that
$$
\l=\tau_0\left(\tfrac{b\mu}{d}\tfrac{1}{1+w}\right)<\tau_0\left(\tfrac{b\mu}{d}\right)=
 \tfrac{b\mu}{d}.
$$
Therefore, $\l\in (0,\tfrac{b\mu}{d})$.
\end{proof}

Since
$$
   \frac{\p}{\p w} \frac{w}{1+w}=\frac{1}{(1+w)^2}\quad \hbox{and}\quad
   \frac{1}{(1+w_0)^2}= \left(\frac{d\l }{b\mu}\right)^2,
$$
the spectrum of the linearization of  \eqref{2.2} at the constant positive solution $w_0$ consists of the eigenvalues of the linear problem
\begin{equation}
\label{2.4}
\left\{
\begin{array}{lll}
-w''=\l w-\frac{d}{b\mu}\l^2w+\tau w\qquad \hbox{in}\;\;(0,1),\\[7pt]
w'(0)=w'(1)=0.
\end{array}
\right.
\end{equation}
According to \eqref{2.3}, the problem \eqref{2.4} has the following sequence of eigencurves
\begin{equation}
\label{2.5}
   \tau_{0,\ell}(\l)\equiv\tau_{0,\ell}(\l,\mu):=\tau_\ell\left(-\l+\tfrac{d}{b\mu}\l^2\right)=\tfrac{d}{b\mu}\l^2-\l+(\ell\pi)^2,\quad \ell\in\mathbb{N}\cup\{0\},
\end{equation}
which is a sequence of polynomials of degree two in $\l$
whose roots are given by
\begin{equation}
\label{2.6}
\l_{\ell}^{\pm}\equiv\l_{\ell}^{\pm}(\mu):=\tfrac{b\mu}{2d}
\left(1\pm\sqrt{1-\tfrac{4d}{b\mu}(\ell\pi)^2}\right),\quad \ell\in\mathbb{N}\cup\{0\}.
\end{equation}
Thus, setting
\begin{equation}
\label{2.7}
\mu_\k:=\frac{d}{b}(2\k\pi)^2,\qquad\k\in\mathbb{N}\cup\{0\},
\end{equation}
it becomes apparent that
\begin{equation*}
\left\{
\begin{array}{ll}
\l_{\kappa}^-=\overline{\l_{\kappa}^+}\in\mathbb{C}\setminus\mathbb{R},\;\; \tau_{0,\kappa}(\l)>0\;\;\hbox{for all}\;\l\in[0,\frac{b\mu}{d}],&\;\hbox{if}\;\,\mu<\mu_\k,\\[15pt]
\l_{\kappa}^-=\l_{\kappa}^+=\frac{b\mu}{2d},&\;\hbox{if}\;\;\mu=\mu_\k,\\[10pt]
\l_{\kappa}^{\pm}\in(0,\frac{b\mu}{d}),\;\;\tau_{0,\kappa}(\l)
\left\{
\begin{array}{ll}
\!\!\!<0\;\;\hbox{if}\;\; \l\in(\l_{\k}^-,\l_{\k}^+),\\[5pt]
\!\!\!>0\;\;\hbox{if}\;\; \l\in(0,\frac{b\mu}{d})\setminus [\l_{\k}^-,\l_{\k}^+],
\end{array}
\right.
&\;\hbox{if}\;\;\mu>\mu_\k.
\end{array}
\right.
\end{equation*}
Moreover, for every integer $\ell\geq0$, the spectral polynomials  $\tau_{0,\ell}(\l)$ satisfy
\begin{equation}
\label{2.8}
\begin{split}
\tau_{0,\ell}(0)=(\ell\pi)^2=\tau_{0,\ell}\left(\tfrac{b\mu}{d}\right)
=\tau_{0,\ell}\left(\tfrac{b\mu}{2d}\right)+\tfrac{b\mu}{4d},
\quad \dot\tau_{0,\ell}\left(\tfrac{b\mu}{2d}\right)=0,
\end{split}
\end{equation}
where $\dot{\tau}_{0,\ell}$ stands for the derivative of $\tau_{0,\ell}(\l)$ with respect to $\l$. Note that, thanks to \eqref{2.8},
$$
\lim_{\mu\uparrow +\infty}\tau_{0,\ell}\left(\tfrac{b\mu}{2d}\right)=-\infty.
$$
The left picture of Figure \ref{Fig2} represents the eigencurves $\tau_{0,\ell}(\lambda)$ for $\ell\in\{0,1,2\}$ when $\mu \in (\mu_1,\mu_2)$.
\begin{figure}[h!]
\centering
\includegraphics[scale=0.41]{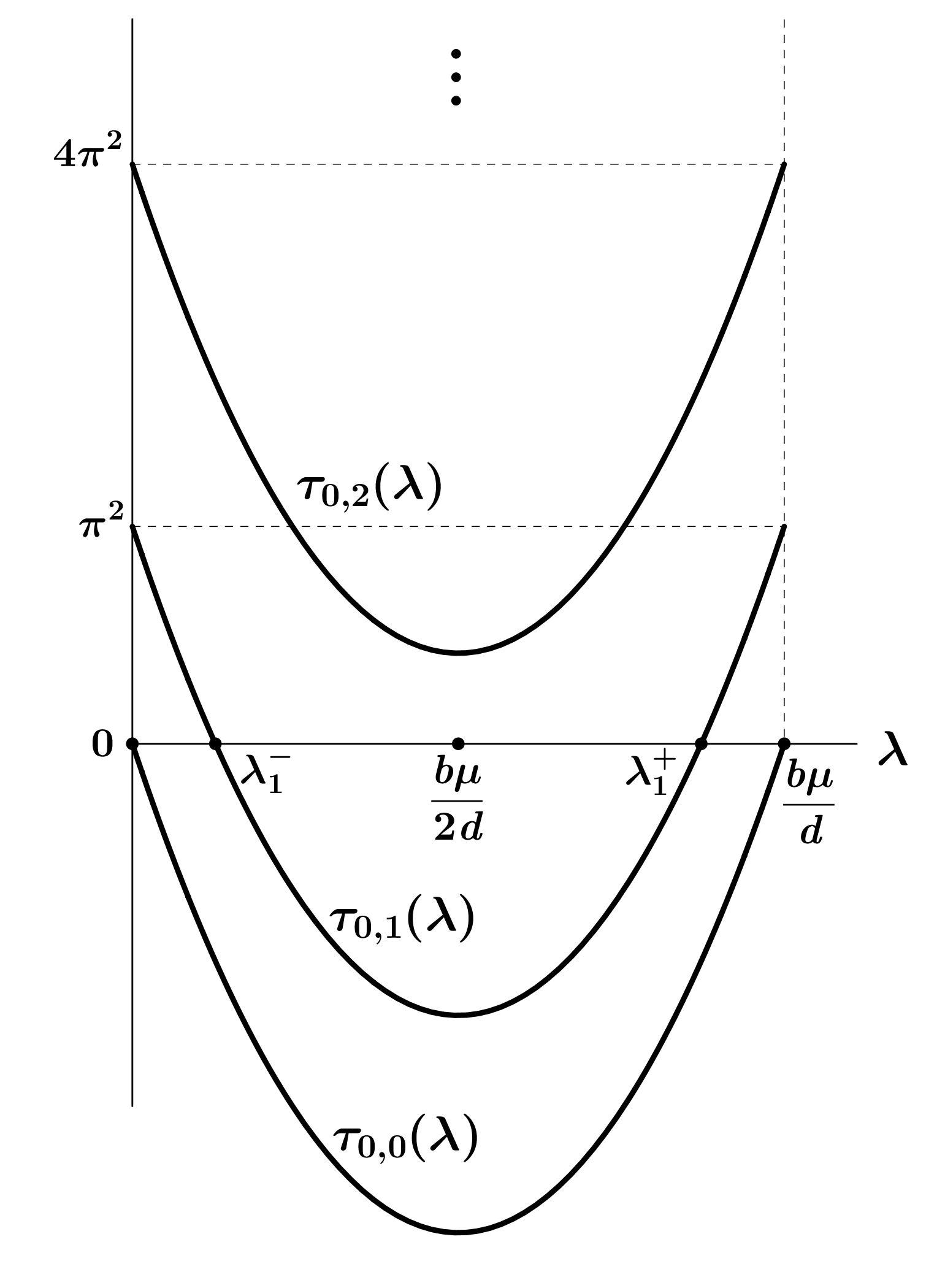}\quad
\includegraphics[scale=0.47]{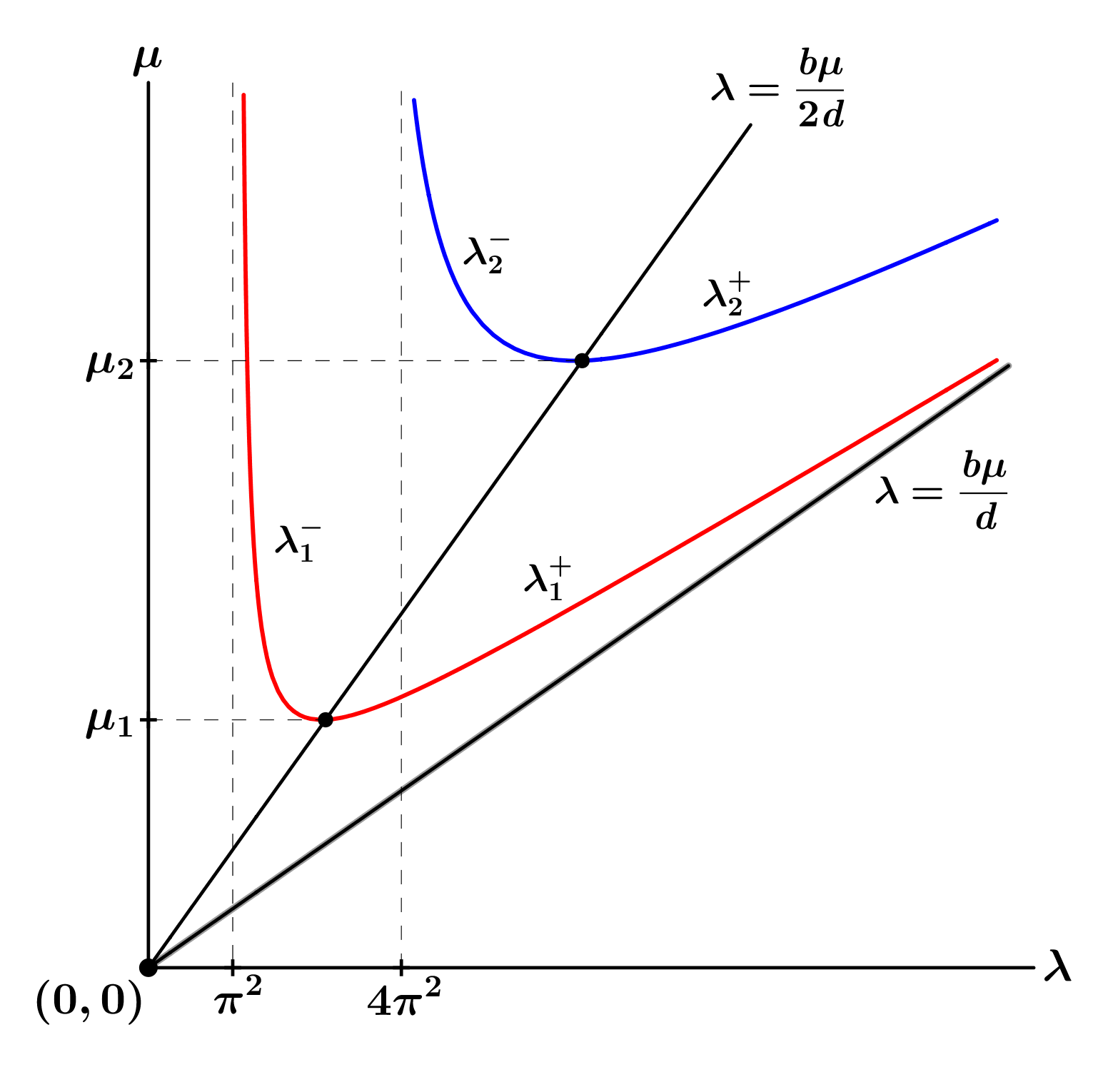}
\caption{The eigencurves $\tau_{0,\k}(\lambda)$, $\k\in\{0,1,2\}$, with $\mu_1<\mu<\mu_2$  (left), and the real roots $\l_1^{\pm}$ and $\l_2^{\pm}$ of the two eigencurves $\tau_{0,1}(\l)$ and $\tau_{0,2}(\l)$, respectively (right).}
\label{Fig2}
\end{figure}
As \eqref{2.7} provides us with an increasing sequence of values of $\mu$ such that
$\lim_{\k\uparrow \infty}\mu_\k =+\infty$, the larger is $\mu$, the larger is  the number of $\k$'s for which $\mu>\mu_\k$ holds. Thanks to \eqref{2.6}, for every $\k\geq1$ and $\mu>\mu_\k$,
\begin{equation}
\label{2.9}
0<\l_{1}^-<\l_{2}^-<\cdots<\l_{\k}^-<\tfrac{b\mu}{2d}<\l_{\k}^+<\cdots<\l_{2}^+<\l_{1}^+<\tfrac{b\mu}{d}.
\end{equation}
Moreover,
\begin{align*}
   \lim_{\mu\uparrow +\infty}\l_{\k}^-& =  \lim_{\mu\uparrow +\infty}
   \left[ \tfrac{b\mu}{2d} \frac{\left(1-\sqrt{1-\tfrac{4d}{b\mu}(\k\pi)^2}\right)
   \left(1+\sqrt{1-\tfrac{4d}{b\mu}(\k\pi)^2}\right)}{1+\sqrt{1-\tfrac{4d}{b\mu}(\k\pi)^2}}\right]\\ & =
   \lim_{\mu\uparrow +\infty}  \left( \tfrac{b\mu}{2d} \frac{ \tfrac{4d}{b\mu} (\k\pi)^2} {1+\sqrt{1-\tfrac{4d}{b\mu}(\k\pi)^2}}\right)=(\k \pi)^2
\end{align*}
and
$$
\lim_{\mu\uparrow +\infty}\frac{\l_{\kappa}^+}{b\mu/d}=1.
$$
In particular,
$$
\lim_{\mu\uparrow +\infty}\l_{\kappa}^+=+\infty.
$$
Furthermore, for every integer $\ell\geq 1$, a direct calculation shows that
$$
    \tfrac{\p\l_{\ell}^-}{\p\mu}<0<\tfrac{\p\l_{\ell}^+}{\p\mu},
$$
which is reinforced by the fact that $\tau_{0,\ell}(\l)$ is decreasing with respect to $\mu$. The right picture in Figure \ref{Fig2} represents the behavior of the roots $\l_1^{\pm}$ and $\l_2^{\pm}$ in the $(\l,\mu)$-plane. Finally, note that, for every $\mu>\mu_\k$, $\l_{\kappa}^{\pm}$ are simple roots of $\tau_{0,\kappa}(\l)$ such that
\begin{equation}
\label{2.x}
\begin{split}
\dot{\tau}_{0,\k}(\l_{\kappa}^-) & = \tfrac{2d}{b\mu}\l_{\kappa}^--1 =-\sqrt{1-\tfrac{4d}{b\mu}(\k\pi)^2}\\ & <0<\sqrt{1-\tfrac{4d}{b\mu}(\k\pi)^2}=\tfrac{2d}{b\mu}\l_{\kappa}^+-1= \dot{\tau}_{0,\kappa}(\l_{\kappa}^+).
\end{split}
\end{equation}
From the previous spectral analysis  one can easily count the number of negative eigenvalues of
the linearization of \eqref{2.2} at $w_0=\frac{b\mu}{d\l}-1$, given by \eqref{2.4}, according to the values of the two parameters $\l\in (0,\frac{b\mu}{d})$ and $\mu>0$. Such a number, coincident with the dimension of the unstable manifold of $w_0$, will be denoted by $\mathscr{M}(w_0)$ throughout this paper, as it is often referred to as the \textit{Morse index}. As, for every $\mu>0$,
$$
  \tau_{0,0}(\l)=\left(\tfrac{d\l}{b\mu}-1\right) \l<0 \;\;\hbox{for all}\;\; \l\in (0,\tfrac{b\mu}{d}),
$$
we have that $\mathscr{M}(w_0)\geq 1$ for all $\mu>0$. As it will become apparent from the next discussion,  the value of $\mathscr{M}(w_0)$ depends on the precise values of $\l$ and $\mu$.
\par
Suppose $\mu <\mu_1$. Then, $\tau_{0, \ell}(\l)>0$ for all integer $\ell\geq 1$ and $\l\in(0,\frac{b\mu}{d})$. Thus, $\mathscr{M}(w_0)=1$ for all $\l\in(0,\frac{b\mu}{d})$. When $\mu=\mu_1$, then,
$\tau_{0, 1}(\tfrac{b\mu}{2d})=0$ and $\tau_{0, 1}(\l)>0$ for all $\l \in (0,\frac{b\mu}{d})\setminus
\{\frac{b\mu}{2d}\}$. Thus, also $\mathscr{M}(w_0)=1$.
\par
Suppose $\mu\in (\mu_1,\mu_2)$. Then, since $\mu<\mu_2$, we have that $\tau_{0,\ell}(\l)>0$ for all $\ell\geq 2$ and $\l\in(0,\frac{b\mu}{d})$. Consequently, $1\leq \mathscr{M}(w_0)\leq 2$ for all 
$\l\in(0,\frac{b\mu}{d})$. Moreover, since
$$
\tau_{0,1}(\l) \left\{ \begin{array}{ll} >0 & \;\;\hbox{if}\;\; \l \in (0,\l_1^-)\cup(\l_1^+,\tfrac{b\mu}{d}),\\ <0 & \;\;\hbox{if}\;\; \l \in (\l_1^-,\l_1^+), \end{array}\right.
$$
it becomes apparent that
\begin{equation}
\label{2.11}
  \mathscr{M}(w_0)=\left\{ \begin{array}{ll} 1 & \;\;\hbox{if}\;\; \l\in (0,\l_1^-]\cup [\l_1^+,\frac{b\mu}{d}),\\[2pt] 2 & \;\;\hbox{if}\;\; \l\in (\l_1^-,\l_1^+).\end{array}\right.
\end{equation}
When $\mu=\mu_2$, we  have that $\tau_{0, 2}(\l)>0$ for each $\l\in(0,\frac{b\mu}{d})
\setminus\{\frac{b\mu}{2d}\}$, but $\tau_{0, 2}(\frac{b\mu}{2d})=0$. Consequently,
\eqref{2.11} remains valid.
\par
More generally, when $\mu \in (\mu_\k,\mu_{\k+1})$ for some integer $\k\geq 2$, then
$$
\tau_{0, \ell}(\l)>0\;\;\hbox{for all}\;\; \ell\geq \k+1\;\; \hbox{and}\;\; \l\in(0,\tfrac{b\mu}{d}),
$$
and
\begin{equation}
\label{2.12}
  \mathscr{M}(w_0)=\left\{ \begin{array}{ll} 1 & \;\;\hbox{if}\;\; \l\in (0,\l_1^-]\cup [\l_1^+,\frac{b\mu}{d}),\\[2pt]  2  & \;\;\hbox{if}\;\; \l\in (\l_1^-,\l_2^-]\cup [\l_2^+,\l_1^+), \\
  \cdots & \\   \k  & \;\;\hbox{if}\;\; \l\in (\l_{\k-1}^-,\l_{\k}^-]\cup [\l_{\k}^+,\l_{\k-1}^+), \\[2pt]
    \k+1  & \;\;\hbox{if}\;\; \l\in (\l_{\k}^-,\l_{\k}^+). \end{array}\right.
\end{equation}
The change of the Morse index by one when $\l$ crosses $\l_\ell^\pm$ for every $\ell\in \{1,....,\k\}$ if $\mu \in (\mu_\k,\mu_{\k+1})$ entails the bifurcation of a  branch of positive solutions of \eqref{2.2} bifurcating
from $\mathscr{C}_0$ at  $\l=\l_\ell^\pm$, $\ell\in \{1,....,\k\}$, as it will become apparent in the forthcoming analysis.

\subsection{Phase plane analysis}

\noindent In this section, we are using  some non-standard phase plane techniques to ascertain the fine global structure of the set of nonconstant positive solutions of \eqref{2.2}.  As a consequence of this analysis, every nonconstant positive solution of \eqref{2.2} will be characterized from its nodal behavior with respect to $w_0$. Precisely, for every integer $n\geq 1$, it is said that   $w_n$ is a positive $(n,w_0)$-nodal solution of \eqref{2.2} if $w_n$ is a positive solution of \eqref{2.2} and $w_n-w_0$ has exactly $n$ zeroes in $(0,1)$, i.e. $w_n$ is a positive solution of \eqref{2.2} crossing $n$ times the constant solution $w_0$.
\par
Since the problem \eqref{2.2} is autonomous, phase plane analysis and map-time techniques can be used.
As its associated first order system is given by
\begin{equation}
\label{2.13}
\left\{
\begin{array}{ll}
w'=z,\\[5pt]
z'=-\l w+\frac{b\mu}{d}\frac{w}{1+w},
\end{array}
\right.
\end{equation}
the  unique nontrivial equilibrium of \eqref{2.13} is $(w_0,0)=(\frac{b\mu}{d\l}-1,0)$, and
the associated total energy is
$$
  \mathscr{E}(w,z):=\tfrac{1}{2}z^2+\tfrac{\l}{2}w^2-\tfrac{b\mu}{d}\left[w-\ln(1+w)\right].
$$
Thus, we can take as potential energy the function
\begin{equation}
\label{2.14}
   F(w):=\tfrac{\l}{2}w^2-\tfrac{b\mu}{d}\left[w-\ln(1+w)\right],\qquad w >-1.
\end{equation}
As the function $F$ satisfies, for every  $w>-1$,
$$
   F'(w)=\l w-\frac{b\mu}{d}\frac{w}{1+w},\qquad  F''(w)=\l-
   \frac{b\mu}{d}\frac{1}{(1+w)^2},
$$
where $'$ stands for the derivative with respect to $w$ here, it is apparent that
$$
(F')^{-1}(0)=\{0,w_0\},\quad F''(0)<0,\quad F''(w_0)>0, \quad \lim_{w\uparrow +\infty}F(w)=+\infty.
$$
Thus, in the $(w,z)$ phase plane, the equilibrium $(0,0)$ is a local saddle point, while the equilibrium $(w_0,0)$ is a local center surrounded by periodic orbits enclosed by the homoclinic connection
$$
  \mathscr{E}(w,z)=\mathscr{E}(0,0)=0.
$$
Therefore, the phase plane portrait of \eqref{2.13} for $w\geq 0$ is the one sketched in Figure \ref{Fig3}.
By simply having a glance at Figure \ref{Fig3} it is easily realized that the positive solutions of \eqref{2.2} should have their trajectories in the negative energy levels of \eqref{2.13}, which have been plotted using green orbits in Figure \ref{Fig3}. Let denote by  $(w_\pm,0)$ the crossing points of a generic periodic orbit with the $w$-axis. Then,
\begin{equation}
\label{2.15}
\begin{split}
\mathscr{E}(w,z)&=\tfrac{1}{2}z^2+\tfrac{\l}{2}w^2-\tfrac{b\mu}{d}\left[w-\ln(1+w)\right]
\\&=\tfrac{\l}{2}w_{\pm}^2-\tfrac{b\mu}{d}\left[w_\pm-\ln(1+w_\pm)\right]=F(w_\pm)<0.
\end{split}
\end{equation}

\begin{figure}[h!]
\centering
\includegraphics[scale=0.22]{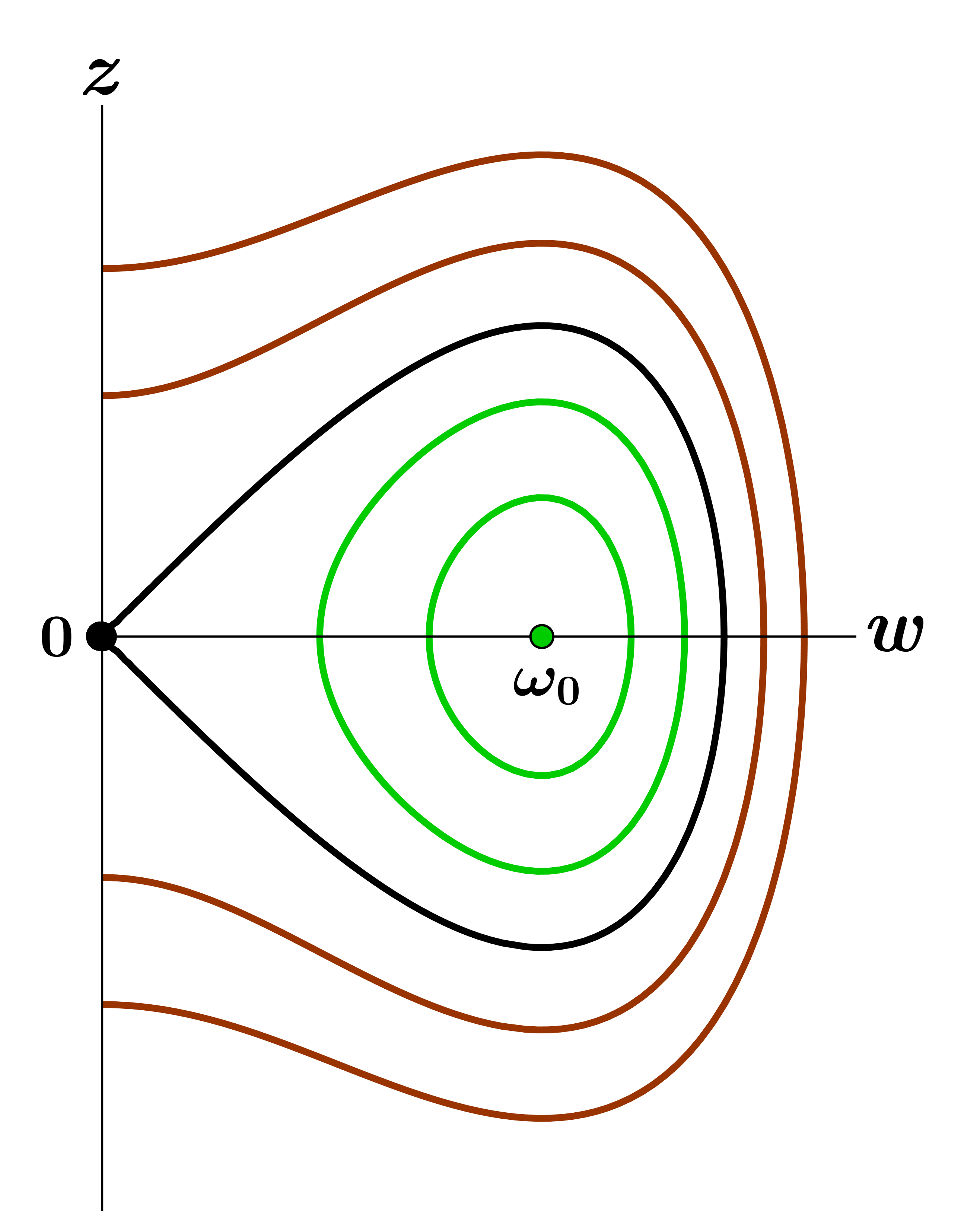}
\caption{Phase plane portrait  of \eqref{2.13}.}
\label{Fig3}
\end{figure}

Subsequently, for every $w_-\in(0,w_0)$, we denote by
$$
   T(w_-):=T(w_-,w_+(w_-))>0
$$
the necessary time used for the solution of \eqref{2.13} to link $(w_-,0)$ with $(w_+,0)$ through the orbit \eqref{2.15}. As  $\mathscr{E}(w,z)=0$ is a homoclinic connection of $(0,0)$, we find that
\begin{equation}
\label{2.16}
\lim_{w_-\downarrow0}T(w_-)=+\infty.
\end{equation}
Moreover, as established by the next lemma, $\lim_{w_-\uparrow w_0}T(w_-)$ can be calculated from the corresponding time map
for the linearized system of \eqref{2.13} at $(w_0,0)$, which is given by
\begin{equation}
\label{2.17}
\left\{
\begin{array}{ll}
w'=z,\\[2pt]
z'=-\l\Big(1-\tfrac{d\l}{b\mu}\Big) w.
\end{array}
\right.
\end{equation}
As it becomes apparent from the proof of the lemma, this is equivalent to compute $T(w_-)$, linearize the potential energy appearing in the $T(w_-)$ formula and letting  $w_-\uparrow w_0$.

\begin{lemma}
\label{le2.2}
For every $w_-\in (0,w_0)$, $T(w_-)$ coincides with half of the period of the periodic orbit of \eqref{2.13} through the point $(w_-,0)$. Moreover,
\begin{equation}
\label{2.18}
T(w_0):=\lim_{w_-\uparrow w_0}T(w_-)=\frac{\pi}{\sqrt{\l\left(1-\tfrac{d\l}{b\mu}\right)}}
\end{equation}
coincides with half of the  period of the corresponding orbit of \eqref{2.17}.
\end{lemma}
\begin{proof}
According to \eqref{2.15}, $T(w_-)$ can be expressed in terms of the potential energy $F(w)$ as
\begin{equation*}
T(w_-)=\int_{w_-}^{w_0}\frac{dw}{\sqrt{2\left[F(w_-)-F(w)\right]}}+
\int_{w_0}^{w_+}\frac{dw}{\sqrt{2\left[F(w_+)-F(w)\right]}}.
\end{equation*}
The fact that $T(w_-)$ equals half of the period of the periodic orbit of \eqref{2.13} through the point $(w_-,0)$ follows from the symmetry of \eqref{2.13} with respect to the $w$-axis.
For the rest of this proof, we set
\begin{align*}
  T_1(w_-) & := \int_{w_-}^{w_0}\frac{dw}{\sqrt{2\left[F(w_-)-F(w)\right]}},\\
  T_2(w_+)  & :=  \int_{w_0}^{w_+}\frac{dw}{\sqrt{2\left[F(w_+)-F(w)\right]}}.
\end{align*}
Then, introducing the changes of variable
$$
   w=w_0+\t(w_--w_0)\quad \hbox{and}\quad w=w_0+\t(w_+-w_0)
$$
in $T_1(w_-)$ and $T_2(w_+)$, respectively, it follows that
\begin{align*}
T_1(w_-) & =-\int_0^1\frac{(w_--w_0)\,d\t}{\sqrt{2\left[F(w_-)-F(w_0+\t(w_--w_0))\right]}},\\
T_2(w_+)& =\int_0^1\frac{(w_+-w_0)\,d\t}{\sqrt{2\left[F(w_+)-F(w_0+\t(w_+-w_0))\right]}}.
\end{align*}
Moreover, for $w_\pm\approx w_0$, the differences inside the square roots can be developed as
\begin{equation*}
\begin{split}
&F(w_\pm)-F(w_0+\t(w_\pm-w_0))\\&=F(w_\pm)-F(w_0)+F(w_0)-F(w_0+\t(w_\pm-w_0))\\
&=F'(w_0)(w_\pm-w_0)+\tfrac{1}{2}F''(w_0)(w_\pm-w_0)^2+O((w_\pm-w_0)^3)\\
&-\left(F'(w_0)\t(w_\pm-w_0)+\tfrac{1}{2}F''(w_0)\t^2(w_\pm-w_0)^2+O((w_\pm-w_0)^3)\right)\\
&=(1-\t^2) \tfrac{1}{2}F''(w_0)(w_\pm-w_0)^2+O((w_\pm-w_0)^3).
\end{split}
\end{equation*}
Thus,
\begin{equation}
\label{2.19}
\begin{split}
\lim_{w_-\uparrow w_0}T(w_-) & =\lim_{w_-\uparrow w_0}T_1(w_-)+\lim_{w_+\downarrow w_0}T_2(w_+)\\ & =2\int_0^1\frac{d\t}{\sqrt{(1-\t^2)F''(w_0)}}=\frac{\pi}{\sqrt{F''(w_0)}}.
\end{split}
\end{equation}
Therefore, taking into account that
$$
  F''(w_0)=\l-   \frac{b\mu}{d}\frac{1}{(1+w_0)^2}=\l\left(1-\frac{d\l}{b\mu}\right),
$$
\eqref{2.18} holds. Finally, since  \eqref{2.17} can be expressed as
$$
   w''+\l \left(1-\tfrac{b\l}{d\mu}\right)  w =0,
$$
it becomes apparent that $T(w_0)$ equals half of the period of the solutions of \eqref{2.17}. The proof is complete.
\end{proof}

Thanks to the continuity of the period map, the next result follows from \eqref{2.16} and \eqref{2.18}.

\begin{proposition}
\label{pr2.1}
Suppose $\mu>\mu_\kappa$ for some integer $\k\geq 1$. Then, for every $\l \in (\l_\k^-,\l_\k^+)$, the problem \eqref{2.2} possesses, at least, two positive $(\k,w_0)$-nodal solutions.
\end{proposition}
\begin{proof} Since $\mu>\mu_\k$, we have that $\tau_{0,\k}(\l)<0$ for all $\l\in (\l_\k^-,\l_\k^+)$. Thus, thanks
to \eqref{2.5}, we find that
$$
   (\k\pi)^2 -\l\left( 1-\tfrac{d\l}{b\mu}\right) <0 \quad \hbox{for all}\;\; \l\in (\l_\k^-,\l_\k^+).
$$
Consequently, thanks to Lemma \ref{le2.2},
\begin{equation}
\label{2.20}
\kappa T(w_0)=\frac{\k\pi}{\sqrt{\l\left(1-\tfrac{d\l}{b\mu}\right)}}<1 \quad \hbox{for all}\;\; \l\in (\l_\k^-,\l_\k^+).
\end{equation}
Hence, by \eqref{2.16}, there exists $w_-\in (0,w_0)$ such that
$$
  \kappa T(w_-)=1.
$$
By construction, the unique solution of the Cauchy problem
\begin{equation*}
\left\{
\begin{array}{ll}
-w''=\lambda w- \frac{b\mu}{d}\frac{w}{1+w},\\[2pt]
w(0)=w_-, \;\; w'(0)=0,
\end{array}
\right.	
\end{equation*}
denoted by $w_\kappa$, provides us with a positive $(\k,w_0)$-nodal solution of \eqref{2.2}. As \eqref{2.2}
is autonomous, considering the extension of $w_\k(x)$ to $[0,2]$,
$$
\hat w_\k(x):=
\left\{
\begin{array}{ll}
w_\kappa(x)&\quad\hbox{if}\;\, x\in[0,1],\\[4pt]
w_\k(2-x)&\quad\hbox{if}\;\, x\in(1,2],
\end{array}
\right.
$$
the shift
\begin{equation}
\label{2.21}
\tilde w_\k(x):=\hat w_\kappa(x+\kappa^{-1}),\qquad x\in[0,1],
\end{equation}
provides us with another positive $(\k,w_0)$-nodal solution of \eqref{2.2}. Note that $w_\k \neq \tilde w_\k$ because
$$
  \tilde w_\k(0)=w_\kappa(\kappa^{-1})\equiv w_+>w_0>w_-=w_\k(0).
$$
This argument shows that positive $(\k,w_0)$-nodal solutions always arise by pairs, and ends the proof.
\end{proof}

However, to  make precise the number and location of the non-constant positive solutions, some monotonicity properties of the period map are needed. Those are direct consequences of the next result of technical nature.

\begin{lemma}
\label{le2.3}
The kinetic of the differential equation of \eqref{2.2},
$$
    f(w):=\l w-\frac{b\mu}{d}\frac{w}{1+w},
$$
is an {\rm A-B}-function, as discussed in Definition 1.4.1 of Schaaf \cite{S90}. Thus, $T(w_-)$ is decreasing with respect to $w_-$, i.e.
\begin{equation}
\label{2.22}
T'(w_-)<0\quad\hbox{for all}\;\; w_-\in(0,w_0),
\end{equation}
where $'$ stands for derivation with respect to $w_-$.
\end{lemma}
\begin{proof}
Let us prove that $f$ is an {\rm A-B}-function in the interval $J=[0,w_h]$, where $w_h>0$ is the unique positive zero of the potential energy $F(w)$. Since
$$
f'(w)=\l-\tfrac{b\mu}{d(1+w)^2}\quad \hbox{for all}\;\; w\in J,
$$
setting $\a:=\sqrt{\tfrac{b\mu}{d\l}}-1$, it becomes apparent that
\begin{equation*}
f'(w)\left\{
\begin{array}{ll}
<0 &\quad\hbox{if}\;\, w\in[0,\a),\\[3pt]
=0 &\quad\hbox{if}\;\, w=\a,\\[3pt]
>0 &\quad\hbox{if}\;\, w>\a.
\end{array}
\right.
\end{equation*}
According to Definition 1.4.1 of Schaaf \cite{S90}, $f$ is an {\rm A-B} function if
\begin{itemize}
\item $f'$ has only simple zeros on $J$;
\item $f$ is an A-function for $w>\a$, i.e.
$$
f'(w)f'''(w)-\tfrac{5}{3}(f''(w))^2<0\;\; \hbox{for all}\;\;w>\a;
$$
\item $f$ is a B-function for $w\in[0,\a)$, i.e.
$$
f(w)f''(w)-3(f'(w))^2\leq0\;\; \hbox{for all}\;\;w\in[0,\a].
$$
\end{itemize}
Since
$$
f''(w)=\frac{2b\mu}{d(1+w)^3}\quad\hbox{and}\quad f'''(w)=-\frac{6b\mu}{d(1+w)^4}\;\;\hbox{for all}\;\; w>0,
$$
we find that $f''(\a)>0$ and
$$
f'(w)f'''(w)-\frac{5}{3}(f''(w))^2=-\frac{6b\mu\l}{d(1+w)^4}+(6-\tfrac{20}{3})\frac{b^2\mu^2}{d^2(1+w)^6}<0
$$
for all $w>0$; in particular, for all $w>\a$. Thus, it remains to check that $f$ is a B-function. Indeed, by the definition of $\a$,
$$
\frac{1}{1+\a}=\sqrt{\frac{d\l}{b\mu}},
$$
and, since we are taking $\l<\tfrac{b\mu}{d}$, it becomes apparent that, for every $w\in[0,\a]$,
\begin{align*}
f(w)f''(w)-3(f'(w))^2 & <f(w)f''(w)=w\left(\l-\tfrac{b\mu}{d}\tfrac{1}{1+w}\right)\tfrac{2b\mu}{d(1+w)^3}\\&\leq
w\left(\l-\tfrac{b\mu}{d}\tfrac{1}{1+\a}\right)\tfrac{2b\mu}{d(1+w)^3}\\
& =w\sqrt{\l}\left(\sqrt{\l}-\sqrt{\tfrac{b\mu}{d}}\right)\tfrac{2b\mu}{d(1+w)^3}<0.
\end{align*}
Therefore, $f(w)$ is an A-B-function. Consequently, by Theorem 2.1.3 of Schaaf \cite{S90} (see also Opial \cite{O61}), \eqref{2.22} holds. The proof is complete.
\end{proof}

\subsection{Proof of Theorem \ref{th1.1}}

\noindent The monotonicity of the time map $T(w_-)$ with respect to $w_-$ established by Lemma \ref{le2.3}, entails  the existence of exactly two positive $(\k,w_0)$-nodal solutions, $w_\k$ and $\tilde w_\k$, for every
$\l\in (\l_{\k}^-,\l_{\k}^+)$ if $\mu>\mu_\k$. Conversely, assume that \eqref{2.2} has a positive  $(\k,w_0)$-nodal solution for some $\l\in (0,\frac{b\mu}{d})$. Then, thanks to Lemmas \ref{le2.3} and \ref{le2.2}, there exists $w_-\in(0,w_0)$ such that
$$
  T(w_-)=\tfrac{1}{\k}>T(w_0)= \tfrac{\pi}{\sqrt{\l\left(1-\tfrac{d\l}{b\mu}\right)}}.
$$
As, according to \eqref{2.5}, this estimate can be equivalently expressed as
$$
\tau_{0, \k}(\l)<0,
$$
necessarily,
$$
\mu>\mu_\k\quad \hbox{and}\quad \l\in (\l_{\k}^-,\l_{\k}^+).
$$
Therefore, thanks to Proposition \ref{pr2.1}, for every integer $\kappa\geq 1$, the problem \eqref{2.2} has a positive $(\k,w_0)$-nodal solution if and only if $\mu>\mu_\k$ and $\l\in (\l_{\k}^-,\l_{\k}^+)$. In such a case, by Lemma \ref{le2.3}, it has a unique positive $(\k,w_0)$-nodal solution, $w$, with $w(0)=w_-<w_0$ and a unique positive $(\k,w_0)$-nodal solution, $\tilde{w}$, with $\tilde{w}(0)=w_+\equiv w(\k^{-1}) >w_0$. Moreover, as, due to \eqref{2.5} and \eqref{2.6},
$$
  \tau_{0,\kappa}(\l_{\k}^\pm)=0\quad \hbox{and}\quad \lim_{\l\to\l_{\k}^\pm}T(w_0)=\tfrac{1}{\kappa},
$$
it becomes apparent that, for every $\mu>\mu_\k$,
$$
     \lim_{\l \downarrow \l_{\k}^-}w_- = w_0=\lim_{\l \uparrow \l_{\k}^+}w_-.
$$
Thus, the positive $(\k,w_0)$-nodal solutions of \eqref{2.2} bifurcate from $w_0$ at $\l=\l_{\k}^\pm$. Actually, since $T(w_-)$ varies continuously with respect to $\l$, also $w_-=w_-(\l)$ varies continuously with respect to $\l$ and, hence, the set of positive $(\k,w_0)$-nodal solutions of \eqref{2.2} consists of a closed loop emanating from $(\l,w_0)$ at $\l=\l_{\k}^\pm$ as it has been illustrated  in Figure \ref{Fig4} for $\kappa =1$ and $\kappa=2$. Subsequently, for every integer $\kappa\geq 1$, this loop will be denoted by $\mathscr{C}_\k$.
As, thanks to \eqref{2.6}, we already know that
$$
    \lim_{\mu\downarrow \mu_\k}\l_{\k}^\pm = \tfrac{b\mu}{2d},
$$
the closed loop $\mathscr{C}_\k\equiv \mathscr{C}_\k(\mu)$ shrinks to the single point
$(\l,w)=(\tfrac{b\mu}{2d},w_0)$ as $\mu\downarrow \mu_\k$. In other words,
$\mathscr{C}_\k$ bifurcates from $w_0=1$ at $\l=\frac{b\mu}{2d}$ as $\mu$ crosses the critical value $\mu_\k$. Note that, indeed, at $\l=\frac{b\mu}{2d}$, one has that
$$
  w_0(\l)= \frac{b\mu}{d\l}-1= 1.
$$
Lastly, since $\mu>\mu_n$ for all $n\in\{1,...,\k\}$ if $\mu >\mu_\k$, for every $n\in\{1,...,\k\}$ the set of positive $(n,w_0)$-nodal solutions of \eqref{2.2} consists of the closed loop $\mathscr{C}_n$, which emanates from $w_0$ at $\l=\l_{n}^\pm$. Note that, thanks to \eqref{2.9}, for every $n\in\{1,\dots,\k-1\}$, the loop $\mathscr{C}_{n+1}$ is nested into $\mathscr{C}_n$, as illustrated in Figure \ref{Fig4}.
\par
With all these ingredients the proof of Theorem \ref{th1.1} is completed easily.

\begin{remark}
\rm If $\k\in 2\mathbb{N}$, then the $(\k,w_0)$-nodal solutions of \eqref{2.2} are $1$-periodic solutions of the differential equation with minimal period $2\kappa^{-1}$. Thus, according to Theorem \ref{th1.1}, for every even integer $n\in\{1,\dots,\k\}$, the loop $\mathscr{C}_n$ consists of periodic solutions of minimal period $2n^{-1}$. In particular, $(\l_{n}^{\pm},w_0)$ is a bifurcation point to periodic solutions from $\mathscr{C}_0$.
\end{remark}

Figure \ref{Fig4} shows the global bifurcation diagram of \eqref{2.2} after Theorem \ref{th1.1} for a fixed $\mu\in(\mu_2,\mu_3]$. Note that, moving the parameter $\mu$ as in the right picture of Figure \ref{Fig2}, when $\mu>\mu_\k$ and $n\in\{1,\dots,\k\}$, the sets of $(n,w_0)$-solutions are structured in paraboloid-like sets, whose vertical sections are the loops $\mathscr{C}_n$.

\begin{figure}[h!]
\centering
\includegraphics[scale=0.6]{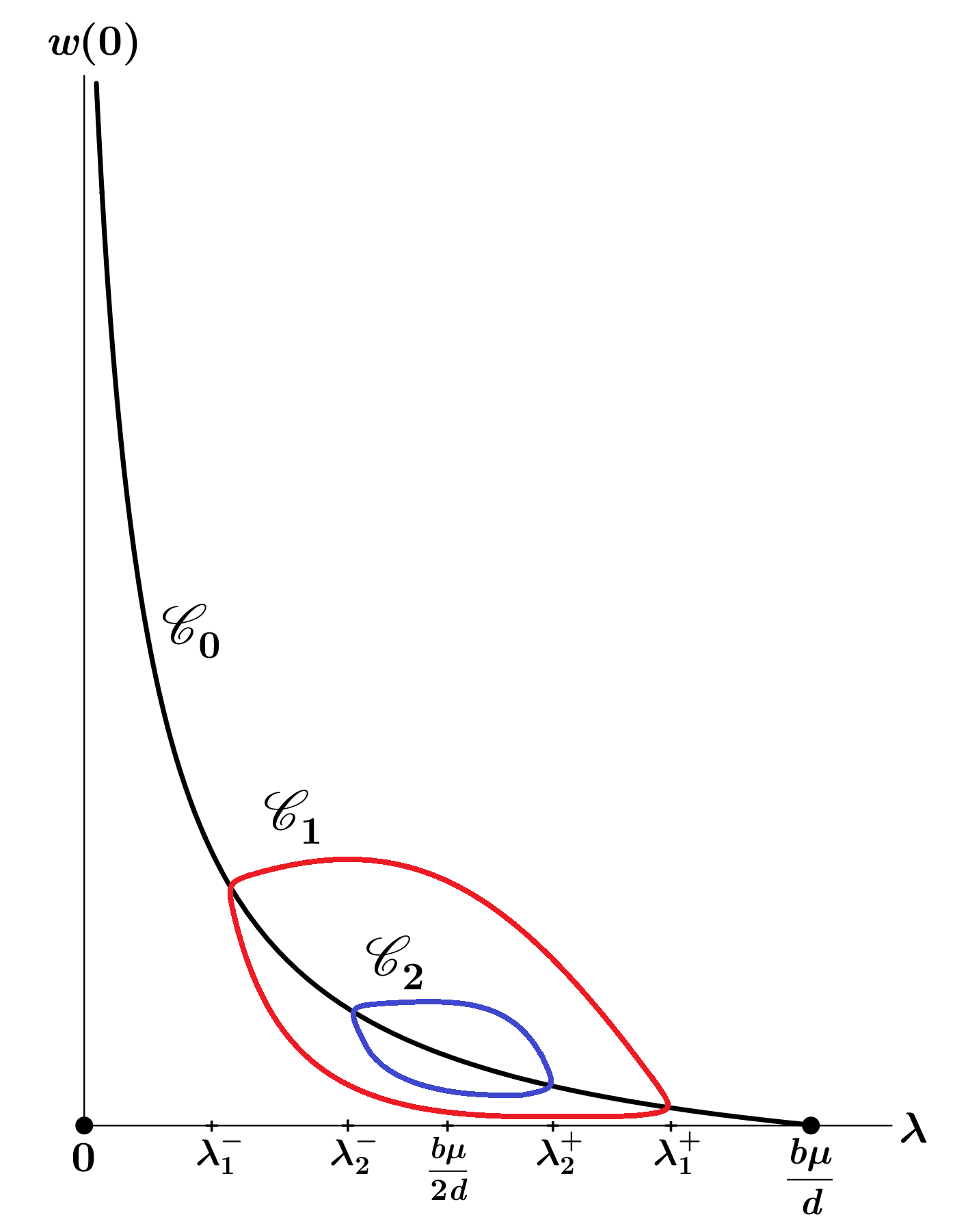}
\caption{The curve of constant solutions $\mathscr{C}_0$ and the loops $\mathscr{C}_1$ and $\mathscr{C}_2$
for $\mu\in(\mu_2,\mu_3]$. As $\mu$ further increases and crosses $\mu_3$, $\mathscr{C}_3$ bifurcates from
$\mathscr{C}_0$ at $\l=\frac{b\mu}{2d}$. And so on and so forth.}
\label{Fig4}
\end{figure}

\setcounter{equation}{0}
\section{Bifurcations of $(n,w_0)$-nodal solutions, $n\geq1$}
\label{sec3}

\noindent This section analyzes  the local bifurcation of $(n,w_0)$-nodal solutions from the constant solution curve $\mathscr{C}_0$. Although, according  to the phase plane analysis of Section \ref{sec2}, we already know that the loops of $(n,w_0)$-nodal solutions, $\mathscr{C}_n$, bifurcate supercritically from $\l_{n}^-$ and subcritically from $\l_{n}^+$, the combined use of the bifurcation theorem by Crandall--Rabinowitz, \cite{CR71}, together with the Exchange Stability Principle \cite{CR73}, provides us, in addition, with the dimensions of the unstable manifolds of these solutions, which are needed to obtain the appropriate non-degeneration conditions for getting the coexistence states of \eqref{1.iii} perturbing from the solutions of \eqref{1.iv}.
\par
To perform the local bifurcation analysis from the positive constant solution curve $w_0$, we make the change of variables
$$
  u:=w-w_0
$$
in \eqref{1.iv}, transforming  $w_0$ into $0$. Then, instead of looking for positive $(n,w_0)$-nodal solutions of \eqref{1.iv}, we are searching for $n$-nodal solutions of the equivalent problem
\begin{equation}
\label{3.i}
\left\{
\begin{array}{ll}
-u''=\l u\left(1-\frac{d\l}{b\mu}\frac{1}{\frac{d\l}{b\mu}u+1}\right)\quad \hbox{in}\;\;(0,1),\\[7pt]
u'(0)=u'(1)=0.
\end{array}
\right.
\end{equation}
Since $w\geq 0$, we have that $u\geq -w_0$ and, hence,
$$
\tfrac{d\l}{b\mu}u+1\geq 1-w_0\tfrac{d\l}{b\mu}=1-\left(\tfrac{b\mu}{d\l}-1\right)\tfrac{d\l}{b\mu}=
\tfrac{d\l}{b\mu}>0.
$$
Thus, the right-hand side of \eqref{3.i} is analytic in $\l$ and $u$.

\subsection{Local bifurcations from $(\l,u)=(\l_{n}^{\pm},0)$, $n\geq 1$.}
\label{subsec3.i}

\noindent Subsequently, for any given integer $j\geq 1$, we denote by $\mc{C}_{N}^j[0,1]$ the Banach space of the functions $u: [0,1]\to\mathbb{R}$ of class $\mc{C}^j$ in $[0,1]$ such that $u'(0)=u'(1)=0$ with the norm of $\mc{C}^j[0,1]$. Then, the solutions of \eqref{3.i} are given by the zeroes of the operator
\begin{align*}
\mathscr{F}\,:(0,\tfrac{b\mu}{d})\,\times\, &\mc{C}_{N}^2[0,1]\longrightarrow \mc{C}[0,1]\\
&\quad(\l,u)\mapsto \mathscr{F}(\l,u):=-u''-\l u\Big(1-\tfrac{d\l}{b\mu}\tfrac{1}{\tfrac{d\l}{b\mu}u+1}\Big),
\end{align*}
which is an analytic Fredholm operator that can be decomposed as
$$
\mathscr{F}(\l,u)=\mathscr{L}(\l)u+\mathscr{N}(\l,u),
$$
where
\begin{equation}
\label{3.ii}
\ms{L}(\l)u:=D_u\ms{F}(\l,0)u=-u''-\l \left(1-\tfrac{d\l}{b\mu}\right)u
\end{equation}
is a linear Fredholm operator of index $0$, and $\ms{N}(\l,u)$ is a nonlinear operator satisfying
$$
 \ms{N}(\l,u)=o(\|u\|)\quad \hbox{as}\;\; \|u\|\to0.
$$
The null-space $N[\ms{L}(\l)]$ is the linear space generated by the nontrivial solutions of the linear problem
\begin{equation}
\label{3.iii}
\left\{
\begin{array}{l}
-\varphi''=\l \left(1-\tfrac{d\l}{b\mu}\right)\varphi\quad \hbox{in}\;\;(0,1),\\[7pt]
\varphi'(0)=\varphi'(1),
\end{array}
\right.
\end{equation}\quad \hbox{in}\;\;(0,1),\\[7pt]
and the generalized spectrum of $\ms{L}(\l)$ is given by
$$
  \Sigma(\ms{L}):=\{\l\in\mathbb{R}\;:\;\dim N[\ms{L}(\l)]\geq1\}.
$$
Recalling \eqref{2.5}, we have that, for every integer $\ell\geq0$,
\begin{equation}
\label{3.iv}
    \l \left(1-\tfrac{d\l}{b\mu}\right)=(\ell\pi)^2-\tau_{0,\ell}(\l).
\end{equation}
Thus, since $(\ell\pi)^2$, $\ell\geq 0$, provides us with the set of eigenvalues of $-D^2$ in $(0,1)$
under Neumann boundary conditions, it becomes apparent that
$$
  \Sigma[\ms{L}(\l)]=\{\l\in\mathbb{R}\;:\; \tau_{0,\ell}(\l)=0\;\;\hbox{for some integer}\;\ell\geq 0\}.
$$
Consequently, according to the analysis of Section \ref{sec2}, we find that $\Sigma[\ms{L}(\l)]=\{0\}$ if $\mu\in(\mu_0,\mu_1)$, and, whenever $\mu\in(\mu_\k,\mu_{\k+1})$ for some integer $\k\geq1$,
\begin{equation}
\label{3.5}
   \Sigma[\ms{L}]:=\{\l_{n}^{\pm}\;:\;1\leq n\leq\k\}\quad\hbox{and}\quad N[\ms{L}(\l_{n}^{\pm})]=
   \mathrm{span\,}[\cos(n\pi x)]
\end{equation}
for all $n\in\{1,\dots,\k\}$.
\par
As $\mathscr{F}(\l,u)$ is analytic, as an application of the Implicit Function Theorem, the set of bifurcation points to nontrivial solutions from the trivial branch $(\l,0)$ is contained in the generalized spectrum  $\Sigma[\mf{L}(\l)]$. Moreover, by the main theorem of  Crandall and Rabinowitz \cite{CR71}, the converse holds also true. Indeed, setting
$$
\dot{\ms{L}}(\l)u:=D_{\l u}\ms{F}(\l,0)u=-\left(1-\tfrac{2d\l}{b\mu}\right)u=
 \dot{\tau}_{0,n}(\l) u
$$
and assuming that $\mu>\mu_\k$ for some integer $\k\geq 1$, the following transversality condition holds
\begin{equation}
\label{3.6}
\dot{\ms{L}}(\l_{n}^{\pm})\left( N[\ms{L}(\l_{n}^{\pm})]\right)\oplus R[\ms{L}(\l_{n}^{\pm})]=\mc{C}[0,1]
\end{equation}
for all $n\in\{1,...,\k\}$,  where we are denoting by $R[T]$ the range, or image,  of any linear operator $T$.
Arguing by contradiction, suppose that \eqref{3.6} fails for some integers $\kappa\geq 1$ and $1\leq n\leq \k$. Then, since $\tau_{0,n}(\l_n^\pm)=0$, it follows from \eqref{3.ii}, \eqref{3.iv} and  \eqref{3.5}, that there exists $\psi\in\mc{C}_{N}^2[0,1]$ such that
$$
   -\psi''-(n\pi)^2\psi=\dot{\tau}_{0,n}(\l_{n}^\pm) \cos(n\pi x).
$$
Multiplying this equation by $\cos(n\pi x)$ and integrating by parts in $[0,1]$, yields
$$
    0=\dot{\tau}_{0,n}(\l_{n}^\pm)\int_0^1\cos^2(n\pi x)\, dx,
$$
which implies $\dot{\tau}_{0,n}(\l_{n}^\pm)=0$ and contradicts \eqref{2.x}. Therefore,
as the transversality condition \eqref{3.6} holds for all $n\in \{1,...,\k\}$ if $\mu >\mu_\k$,
the next result holds  from the main theorem of Crandall and Rabinowitz \cite{CR71}.

\begin{theorem}
\label{th3.i}
Suppose $\mu>\mu_\k$ for some $\k\geq 1$. Then, for every integer $n\in\{1,...,\k\}$,
\begin{equation}
\label{3.7}
\left\{ \begin{array}{l}
N[\ms{L}(\l_{n}^{\pm})]=\mathrm{span\,}[\cos(n\pi x)],\\[5pt]
\dot{\ms{L}}(\l_{n}^{\pm})\left(N[\ms{L}(\l_{n}^{\pm})]\right)\oplus R[\ms{L}(\l_{n}^{\pm})]=\mc{C}[0,1].
\end{array}\right.
\end{equation}
Moreover, since
$$
  Y_n:=\left\{ y \in \mc{C}^2_N[0,1]\;:\; \int_0^1 y(x)\cos(n\pi x)\,dx =0\right\}
$$
is a closed subspace of $\mc{C}_{N}^2[0,1]$ such that
\begin{equation}
\label{3.8}
    N[\ms{L}(\l_{n}^{\pm})]\oplus Y_n=\mc{C}^2_N[0,1],
\end{equation}
there exist $\varepsilon>0$ and four (unique) analytic maps
$$
    \l_{n}^\pm:(-\e,\e)\rightarrow\mathbb{R},\qquad y_{n}^\pm:(-\e,\e)\rightarrow Y_n,
$$
such that $\l^\pm_{n}(0)=\l_{n}^{\pm}$, $y_{n}^\pm(0)=0$, and, for every $s\in(-\e,\e)$,
$$
    \ms{F}(\l_{n}^\pm(s),u_{n}^\pm(s))=0,\;\;\hbox{where}\;\; u_{n}^\pm(s):=s\left(\cos(n\pi x)+y_{n}^\pm(s)\right).
$$
Furthermore, there is  a neighborhood of $(\l_{n}^{\pm},0)$ in $\mathbb{R}\times \mc{C}_N^2[0,1]$,
denoted by $\ms{U}_{n}^\pm$, such that
$$
 (\l,u)=(\l_{n}^\pm(s),u_{n}^\pm(s)) \;\; \hbox{for some}\;\; s\in (-\e,\e)\;\;\hbox{if} \;\;\left\{
\begin{array}{ll} (\l,u)\in \ms{U}_{n}^\pm\cap\mathscr{F}^{-1}(0),\\[3pt]
(\l,u)\neq(\l,0).
\end{array}
\right.
$$
\end{theorem}
\begin{proof}
In the previous discussion,
\eqref{3.7} has been already shown if $\mu>\mu_\k$, and \eqref{3.8}
is obvious from \eqref{3.7}. The remaining assertions of the theorem are direct consequences of the main theorem of \cite{CR71}.
\end{proof}

Note that, for sufficiently small $s\in (-\e,\e)$,
$$
     (\l,u)=(\l_{n}^\pm(s),u_{n}^\pm(s))
$$
is a solution of \eqref{3.i} with exactly $n$ zeroes in $(0,1)$. Therefore,
$$
     (\l,w)=(\l_{n}^\pm(s),w_{n}^\pm(s)), \quad w_{n}^\pm(s):=w_0+u_{n}^\pm(s),
$$
provides us with a non-constant positive solution of \eqref{1.iv} crossing $w_0$ exactly $n$ times.

\subsection{Bifurcation directions from $(\l,u)=(\l_{n}^{\pm},0)$, $n\geq 1$.}
\label{subsec3.ii}

\noindent Suppose that $\mu>\mu_\k$, $\k\in\mathbb{N}$, and, for every integer $n\in\{1,...,\k\}$, let
$$
    \l^\pm(s)\equiv \l^\pm_{n}(s),\quad u^\pm(s)\equiv u_{n}^\pm(s),
$$
the analytic maps given by Theorem \ref{th3.i}. These functions admit asymptotic expansions of the type
\begin{equation}
\label{3.9}
\begin{split}
&\l^\pm(s):=\l_{n}^{\pm}+\eta_1^\pm s+\eta_2^\pm s^2+O(s^3),\\[3pt]
&u^\pm(s):=s[\cos(n\pi x)+y_1^\pm s+y_2^\pm s^2+O(s^3)],
\end{split}
\end{equation}
for some $\eta_1^\pm$, $\eta_2^\pm\in\mathbb{R}$ and $y_1^\pm$, $y_2^\pm \in Y_n$. The next result holds.
\begin{proposition}
\label{pr3.i}
Suppose $\mu>\mu_\k$ for some integer $\k\geq 1$, and $n\in\{1,...,\k\}$. Then,
\begin{equation}
\label{3.viii}
\eta_1^{\pm}=0,\qquad \eta_2^->0>\eta_2^+.
\end{equation}
Therefore, the $n$-nodal solutions bifurcate supercritically from  $(\l_{n}^-,0)$, and subcritically
from $(\l_{n}^+,0)$, in agreement with Theorem \ref{th1.1}.
\end{proposition}

\begin{proof}
By setting $\varphi_n:=\cos(n\pi\cdot)$, substituting  \eqref{3.9} into \eqref{3.i}, and dividing by $s$ the resulting identity, yields
\begin{equation}
\label{3.11}
\begin{split}
-[\varphi_n&+y_1^{\pm}s+y_2^{\pm}s^2+O(s^3)]''\\&=( \l_{n}^{\pm}+\eta_1^{\pm}s+\eta_2^{\pm}s^2+O(s^3))
\left[\varphi_n+y_1^{\pm}s+y_2^{\pm}s^2+O(s^3)\right]g(s),
\end{split}
\end{equation}
where
\begin{equation*}
g(s):=1-\frac{d\left( \l_{n}^{\pm}+\eta_1^{\pm}s+\eta_2^{\pm}s^2+O(s^3)\right)}{b\mu
\left(\frac{d( \l_{n}^{\pm}+\eta_1^{\pm}s+\eta_2^{\pm}s^2+O(s^3))}{b\mu}s
\left(\varphi_n+y_1^{\pm}s+y_2^{\pm}s^2+O(s^3)\right)+1\right)}
\end{equation*}
for sufficiently small $|s|$. Note that, for $s\simeq 0$,
\begin{equation}
\label{iii.12}
g(s)= 1-\frac{d}{b\mu}\left( \l_{n}^{\pm}+\eta_1^{\pm}s+\eta_2^{\pm}s^2+O(s^3)\right)h(s),
\end{equation}
where
\begin{align*}
h(s) & := \frac{1}{1+s\frac{d}{b\mu}\left( \l_{n}^{\pm}+\eta_1^{\pm}s+\eta_2^{\pm}s^2+O(s^3)\right)
\left(\varphi_n+y_1^{\pm}s+y_2^{\pm}s^2+O(s^3)\right)}\\[5pt] & = \frac{1}{1+\frac{d\l_{n}^\pm}{b\mu}\v_ns+
  \left(\frac{d\l_{n}^\pm}{b\mu}y_1^\pm+\frac{d\eta_1^\pm}{b\mu}\v_n\right)s^2+O(s^3)}.
\end{align*}
Moreover, for every $p$, $q\in\mathbb{R}$,
\begin{align*}
 \frac{1}{1+p s+q s^2+O(s^3)}& = 1-\left( p s+qs^2 +O(s^3)\right) +\left( p s+qs^2 +O(s^3)\right)^2+O(s^3) \\ & = 1-ps+\left(p^2-q\right)s^2+O(s^3)\qquad \hbox{as}\;\;s \to 0.
\end{align*}
Thus, for $s\simeq 0$,
\begin{equation}
\label{iii.13}
  h(s)= 1-\tfrac{d\l_{n}^\pm}{b\mu}\v_n s+\Big[\Big(\tfrac{d\l_{n}^\pm}{b\mu}\v_n\Big)^2-
\tfrac{d\l_{n}^\pm}{b\mu}y_1^\pm-\tfrac{d\eta_1^\pm}{b\mu}\v_n \Big]s^2 +O(s^3).
\end{equation}
Consequently, substituting \eqref{iii.13} into \eqref{iii.12} yields to
\begin{align*}
  g(s) = 1 & -\tfrac{d\l_{n}^\pm}{b\mu}+\Big[\Big(\tfrac{d\l_{n}^\pm}{b\mu}\Big)^2\v_n -
  \tfrac{d\eta_1^\pm}{b\mu} \Big]s \\ & + \Big[-\tfrac{d\l_{n}^\pm}{b\mu}
  \Big(\Big(\tfrac{d\l_{n}^\pm}{b\mu}\v_n\Big)^2-
   \tfrac{d\l_{n}^\pm}{b\mu}y_1^\pm-\tfrac{d\eta_1^\pm}{b\mu}\v_n \Big) \\
   & \hspace{3.8cm} +\tfrac{d\eta_1^\pm}{b\mu} \tfrac{d\l_{n}^\pm}{b\mu}\v_n -\tfrac{d\eta_2^\pm}{b\mu}
   \Big]s^2+O(s^3)\quad \hbox{as}\;\; s\to 0.
\end{align*}
Equivalently,
\begin{equation}
\label{iii.14}
\begin{split}
  g(s)  & = 1  -\tfrac{d\l_{n}^\pm}{b\mu}+\Big[\Big(\tfrac{d\l_{n}^\pm}{b\mu}\Big)^2\v_n -
  \tfrac{d\eta_1^\pm}{b\mu} \Big]s \\ &\; + \!\Big[\! -\!\Big(\tfrac{d\l_{n}^\pm}{b\mu}\Big)^3\v_n^2+
   \Big( \tfrac{d\l_{n}^\pm}{b\mu}\Big)^2 y_1^\pm +2 \tfrac{d\eta_1^\pm}{b\mu} \tfrac{d\l_{n}^\pm}{b\mu}\v_n -\tfrac{d\eta_2^\pm}{b\mu}
   \Big]s^2\!+\!O(s^3)
\end{split}
\end{equation}
as $s\to 0$. Now, going back to \eqref{3.11} and evaluating \eqref{3.11} at $s=0$ shows that
\begin{equation}
\label{iii.15}
    -\varphi_n''=\l_{n}^{\pm}\left(1-\tfrac{d\l_{n}^{\pm}}{b\mu}\right)\v_n=(n\pi)^2\varphi_n,
\end{equation}
which holds true by the definition of $\varphi_n$ and $\tau_{0,n}(\l)$ in \eqref{2.5}. Moreover, dividing by $s$ the identity \eqref{3.11} and  letting $s\to 0$ it follows from \eqref{iii.15}  that
\begin{equation}
\label{iii.16}
\begin{split}
\left[-D^2-(n\pi)^2\right]y_1^{\pm}&=-(y_1^{\pm})''-\l_{n}^{\pm}\left(1-\tfrac{d\l_{n}^{\pm} }{b\mu}\right)y_1^{\pm}\\&=\eta_1^{\pm}\varphi_n
\Big(1-2\tfrac{d\l_{n}^{\pm}}{b\mu}\Big)+\l_{n}^{\pm}
\left(\tfrac{d\l_{n}^{\pm}}{b\mu}\right)^2\varphi_n^2.
\end{split}
\end{equation}
Since
$$
   \int_0^1\varphi_n^2=\tfrac{1}{2},\qquad \int_0^1\varphi_n^3=0,
$$
multiplying \eqref{iii.16} by $\v_n$ and integrating by parts in $[0,1]$, it follows from
\eqref{iii.15} that
$$
0=\eta_1^{\pm}\left(1-2\tfrac{d\l_{n}^{\pm}}{b\mu}\right)=-\eta_1^{\pm}\dot{\tau}_{0,n}(\l_{n}^{\pm}).
$$
Therefore, owing to \eqref{2.x},
$$
\eta_1^{\pm}=0,
$$
and, hence, thanks to Theorem \ref{th3.i} and \eqref{iii.16}, $y_1^{\pm}$ is the unique solution of
\begin{equation}
\label{iii.17}
\left\{
\begin{array}{l}
 \left[-D^2-(n\pi)^2\right]y_1^{\pm} = \l_{n}^{\pm}
\left(\tfrac{d\l_{n}^{\pm}}{b\mu}\right)^2\varphi_n^2\varphi\quad \hbox{in}\;\;(0,1),\\[7pt]
(y_1^{\pm})'(0)=(y_1^{\pm})'(1)=0,
\end{array}\right.
\end{equation}
such that
\begin{equation}
\label{iii.18}
    \int_0^1\varphi_ny_1^{\pm}=0.
\end{equation}
Since
$$
\varphi_n^2(x)=\cos^2(n\pi x)=\tfrac{1}{2}[1+\cos(2n\pi x)],
$$
it is easily seen that
$$
p^{\pm}(x):=\tfrac{\l_{n}^{\pm}}{2}
\left(\tfrac{d\l_{n}^{\pm}}{n\pi b\mu}\right)^2 [\tfrac{1}{3}\cos(2n\pi x)-1]
$$
provides us with a particular solution of the differential equation of \eqref{iii.17}. Thus,
$$
y_1^{\pm}(x)=A\cos(n\pi x)+\tfrac{\l_{n}^{\pm}}{2}
\left(\tfrac{d\l_{n}^{\pm}}{n\pi b\mu}\right)^2[\tfrac{1}{3}\cos(2n\pi x)-1],
$$
with $A\in\mathbb{R}$, provides us with the set of solutions of \eqref{iii.17}. Lastly, by \eqref{iii.18}, it follows that
$$
0=\int_0^1 \varphi_ny_1^{\pm}=A\int_0^1\cos^2(n\pi x)\,dx.
$$
Hence, $A=0$ and
\begin{equation}
\label{3.19}
y_1^{\pm}(x)=p^{\pm}(x)=\tfrac{\l_{n}^{\pm}}{2}
\left(\tfrac{d\l_{n}^{\pm}}{n\pi b\mu}\right)^2[\tfrac{1}{3}\cos(2n\pi x)-1].
\end{equation}
Now, taking into account \eqref{iii.15}, \eqref{iii.16}, and $\eta_1^{\pm}=0$, dividing \eqref{3.11} by $s^2$ and letting $s\to 0$ in the resulting identity, it follows from \eqref{iii.14} that
\begin{equation}
\label{3.20}
\begin{split}
\left(-D^2-(n\pi)^2\right)y_2^{\pm} & =-(y_2^{\pm})''-\l_{n}^{\pm}
\left(1-\tfrac{d\l_{n}^{\pm}}{b\mu}\right)y_2^{\pm}
\\&=\eta_2^{\pm}\varphi_n\left(1-2\tfrac{d\l_{n}^{\pm}}{b\mu}\right)
+2\l_{n}^{\pm}\left(\tfrac{d\l_{n}^{\pm}}{b\mu}\right)^2\varphi_ny_1^{\pm}
\\ & \hspace{3.4cm} -\l_{n}^{\pm}\left(\tfrac{d\l_{n}^{\pm}}{b\mu}\right)^3\varphi_n^3.
\end{split}
\end{equation}
Multiplying by $\varphi_n$ and integrating by parts in \eqref{3.20} yields
\begin{equation}
\label{3.21}
\frac{\eta_2^{\pm}}{2}\dot{\tau}_{0,n}(\l_{n}^{\pm})=
2\l_{n}^{\pm}\left(\tfrac{d\l_{n}^{\pm}}{b\mu}\right)^2
\int_0^1\varphi_n^2y_1^{\pm}-\l_{n}^{\pm}\left(\tfrac{d\l_{n}^{\pm}}{b\mu}\right)^3\int_0^1\varphi_n^4,
\end{equation}
where $\dot{\tau}_{0,n}(\l_{n}^{\pm})$ comes from \eqref{2.x}. Since $\int_0^1\varphi_n^4>0$ and, thanks to \eqref{3.19},
\begin{align*}
\int_0^1\varphi_n^2y_1^{\pm} \,dx & = \tfrac{\l_{n}^{\pm}}{2}
\left(\tfrac{d\l_{n}^{\pm}}{n\pi b\mu}\right)^2\int_0^1 \cos^2(n\pi x)[\tfrac{1}{3}\cos(2n\pi x)-1]\, dx \\ & =\tfrac{\l_{n}^{\pm}}{4}
\left(\tfrac{d\l_{n}^{\pm}}{n\pi b\mu}\right)^2 \int_0^1[1+\cos(2n\pi x)][\tfrac{1}{3}\cos(2n\pi x)-1]\, dx\\&=\tfrac{\l_{n}^{\pm}}{4}
\left(\tfrac{d\l_{n}^{\pm}}{n\pi b\mu}\right)^2 \left(\tfrac{1}{3} \int_0^1\cos^2(2n\pi x)\, dx-1\right) \\ & =-\tfrac{5\l_{n}^{\pm}}{24}\left(\tfrac{d\l_{n}^{\pm}}{n\pi b\mu}\right)^2<0,
\end{align*}
it follows from \eqref{3.21} that
$$
  \mathrm{sign\,}\eta_2^{\pm}=-\mathrm{sign\,}\dot{\tau}_{0,n}(\l_{n}^{\pm}).
$$
Therefore, by \eqref{2.x}, the proof of \eqref{3.viii} is complete.
\end{proof}

\subsection{Dimensions of the unstable manifolds}

\noindent Thanks to Theorem \ref{th3.i} and the Exchange Stability Principle of Crandall and Rabinowitz
\cite{CR73}, whenever $\mu\in (\mu_\k,\mu_{\k+1})$, the dimensions of the unstable manifolds of the associated solutions of \eqref{2.2}
\begin{equation*}
     (\l,w)= (\l_{n}^\pm(s),w_0+u_{n}^\pm(s)), \quad 0<|s|<\e,\quad n\in\{1,...,\k\}
\end{equation*}
can be determined  from the dimensions of the unstable manifolds, $\ms{M}(w_0)$,  of $(\l,w_0)$ already collected in \eqref{2.12}. This is the main goal of this section.
\par
For every $n\in\{1,...,\k\}$ and sufficiently small $|s|>0$, the spectrum of the linearization of \eqref{2.2} at
\begin{equation}
\label{3.22}
    (\l_{n}^\pm(s),w_{n}^\pm(s)) \equiv (\l_{n}^\pm(s),w_0+u_{n}^\pm(s)), \;\;
0<|s|<\e,\;\; n\in\{1,...,\k\},
\end{equation}
consists of the eigenvalues of the linear problem
\begin{equation}
\label{3.23}
\left\{
\begin{array}{ll}
  \big[-D^2-\l_{n}^\pm +\frac{b\mu}{d}\frac{1}{(1+w_{n}^\pm )^2}\big]\varphi=\tau \varphi\quad \hbox{in}\;\;(0,1),\\[5pt]
\varphi'(0)=\varphi'(1)=0.
\end{array}
\right.
\end{equation}
More generally, for any given  $(n,w_0)$-solution, $(\l,w_{n})$, of \eqref{2.2} with $n\in \{1,...,\k\}$, the spectrum of the linearization of \eqref{2.2} at $(\l,w_{n})$ consists of the eigenvalues of the problem
\begin{equation}
\label{3.24}
\left\{
\begin{array}{ll}
\big[-D^2-\l+\frac{b\mu}{d}\frac{1}{(1+w_{n})^2}\big]\varphi=\tau \varphi\quad \hbox{in}\;\;(0,1),\\[7pt]
\varphi'(0)=\varphi'(1)=0.
\end{array}
\right.
\end{equation}
According to Theorem \ref{th1.1}, $(\l,w_{n})$ exists if and only if $\l\in(\l_{n}^-,\l_{n}^+)$. Actually, there are exactly two solutions of this type for each $\l\in(\l_{n}^-,\l_{n}^+)$.
It is folklore that \eqref{3.24} has an increasing sequence of algebraically simple eigenvalues
\begin{equation}
\label{3.25}
\tau_{n,\ell}(\l):=\tau_\ell\left(-\l+\tfrac{b\mu}{d}\tfrac{1}{(1+w_{n})^2}\right),\qquad\ell\geq0,
\end{equation}
whose associated eigenfunctions have exactly $\ell$ zeroes in $(0,1)$, respectively. Since the change of variables
$$
   w_{n}=w_0+u_{n}
$$
is linear, the eigenvalues \eqref{3.25} coincide with the ones of the linearization of \eqref{3.i} at $u_{n}$. The next  result collects the main properties of these eigenvalues.

\begin{theorem}
\label{th3.2}
Suppose $\mu\in(\mu_\k,\mu_{\k+1})$ for some integer $\k\geq1$, and $n\in\{1,\dots,\k\}$. Then, the following properties are satisfied:
\begin{enumerate}
\item[{\rm (a)}] There exists $\d>0$ such that, for every
$$
     \l\in \Lambda_n:=(\l_{n}^-,\l_{n}^-+\d)\cup(\l_{n}^+-\d,\l_{n}^+),
$$
the following estimates hold
\begin{equation}
\label{3.26}
     \tau_{n,0}(\l)<\cdots<\tau_{n,n-1}(\l)<0<\tau_{n,n}(\l).
\end{equation}
Equivalently, for sufficiently small $|s|$, the linearization of \eqref{2.2} at 
any solution $(\l_{n}^\pm(s),w_{n}^\pm(s))$ has, exactly, $n$ negative eigenvalues, i.e. the Morse
index of $w_{n}^\pm(s)$, denoted by $\ms{M}(w_{n}^\pm(s))$, satisfies
\begin{equation}
\label{3.27}
    \ms{M}(w_{n}^\pm(s))=n=\ms{M}(w_0)-1,
\end{equation}
where $\ms{M}(w_0)$ is the Morse index of $w_0$ at $\lambda\in\Lambda_{n}$ (see \eqref{2.12}).
\par
\item[{\rm (b)}] For every $\l\in(\l_{n}^-,\l_{n}^+)$,
\begin{equation}
\label{3.28}
     \tau_{n,n-1}(\l)\leq0\leq\tau_{n,n}(\l).
\end{equation}
\item[{\rm (c)}] $\tau_{n,n-1}(\l)$ and $\tau_{n,n}(\l)$ vanish, at most, finitely many times
in $(\l_{n}^-,\l_{n}^+)$.
\end{enumerate}
\end{theorem}
\begin{proof}
The inclusion operator $J: \mc{C}_{N}^2[0,1]\hookrightarrow \mc{C}[0,1]$ satisfies
\begin{equation}
\label{3.29}
J\v_n =\v_n \notin R[\ms{L}(\l_{n}^{\pm})],
\end{equation}
where $\v_n(x)=\cos(n\pi x)$ is the generator of $N[\ms{L}(\l_{n}^{\pm})]$. Indeed, according to
Section \ref{subsec3.i}, if \eqref{3.29} is not true, then there exists $\psi\in\mc{C}_{N}^2[0,1]$ such that
$$
-\psi''-(n\pi)^2\psi =\varphi_n.
$$
Multiplying this identity  by $\varphi_n$ and integrating by parts in $[0,1]$ gives a contradiction.
Thus, \eqref{3.29} holds. Moreover, by \eqref{3.ii} and \eqref{3.iv},  we have that
$$
\ms{L}(\l)\varphi_n=[(n\pi)^2-\l(1-\tfrac{d\l}{b\mu})]\varphi_n.
$$
So,
$$
a(\l):=(n\pi)^2-\l(1-\tfrac{d\l}{b\mu})=\tau_{0,n}(\l)
$$
provides us with the perturbed eigenvalue of $\ms{L}(\l)$ from  $a(\l_{n}^\pm)=0$; the perturbed
eigenfunction is $\v_n$ for all $\l \in (0,\tfrac{b\mu}{d})$. In general, these (analytic) perturbations
are guaranteed by Lemma 2.4.1(1) of \cite{LG01}. Since
$$
   \ms{G}(\l,w):=-w''-\lambda w+ \frac{b\mu}{d}\frac{w}{1+w},\quad (\l,w)\in \mathbb{R}\times \mc{C}_N^2[0,1], \qquad w\geq -\tfrac{1}{2},
$$
is analytic, by Lemma 2.4.1(2) of \cite{LG01}, there exist $\e>0$ and two (unique) analytic maps
$\varrho^\pm:(-\e,\e)\to \mathbb{R}$ and $\Phi^\pm:(-\e,\e)\to \mc{C}_N^2[0,1]$ such that
$\varrho^\pm(0)=0$, $\Phi^\pm(0)=\v_n$, and, for every $s\in(-\e,\e)$,
$$
\Phi^\pm(s)-\v_n \in Y_n,\quad D_u\ms{G}(\l_{n}^\pm(s),w_{n}^\pm(s))\Phi^\pm(s)=\varrho^\pm(s)\Phi^\pm(s),
$$
where $Y_n$ stands for the $L^2$-orthogonal of $\v_n$ in $\mc{C}_N^2[0,1]$. Thus, $\varrho^\pm(s)$ is the perturbed eigenvalue of $0$ along the bifurcating branch $(\l_{n}^\pm(s),w_{n}^\pm(s))$, $|s|<\e$. According to the Exchange Stability Principle
of Crandall and Rabinowitz \cite{CR73} (see Theorem 2.4.2 of \cite{LG01}),
$$
\lim_{\substack{s\to 0\\
\varrho^{\pm}(s)\neq 0}}
\frac{-s\dot \l_{n}^\pm(s)\dot a(\l_{n}^\pm)}{\varrho^\pm(s)}=1.
$$
In particular, the function $\varrho^\pm(s)$ has the same zeroes as $-s \dot \l_{n}^\pm(s) \dot a(\l_{n}^\pm)$ and, whenever $\varrho^\pm(s)\neq 0$, the same sign. Thus, since $\eta_1^\pm=0$ and
$$
    \dot a(\l_{n}^{\pm})=\dot\tau_{0,n}(\l_{n}^{\pm}),\quad \dot \l_{n}^\pm(s) =2\eta_2^\pm s + O(s^2),
$$
we find that
$$
\mathrm{sign\,} \varrho^\pm(s) = \mathrm{sign\,}\left(-s^2\eta_2^\pm \dot\tau_{0,n}(\l_{n}^{\pm})\right)
$$
for sufficiently small $|s|$. Consequently, as
$$
\mathrm{sign\,} \dot\tau_{0,n}(\l_{n}^{\pm})=\pm 1,
$$
it becomes apparent that
$$
  \mathrm{sign\,} \varrho^\pm(s) = -\mathrm{sign\,}(\pm \eta_2^\pm).
$$
As, due to Proposition \ref{pr3.i}, $\eta_2^->0>\eta_2^+$, we can conclude that $\varrho^\pm(s)>0$ for sufficiently 
small $|s|$. Therefore, $\ms{M}(w_{n}^\pm(s))=n$. According to \eqref{2.12}, this ends the proof of \eqref{3.27}.
\par
The proof of the first assertion of Part (a) follows readily from Theorem \ref{th1.1} and the uniqueness assertion in
Theorem \ref{th3.i}, as the curve $(\l_{n}^-(s),w_{n}^-(s))$, $s\approx 0$, provides us with the two
$(n,w_0)$-nodal solutions of \eqref{1.iv} in a right-neighborhood of $\l_{n}^-$, and, similarly, $(\l_{n}^+(s),w_{n}^+(s))$, $s\approx 0$, are the two  $(n,w_0)$-nodal solutions of \eqref{1.iv} in a left-neighborhood of $\l_{n}^+$. This concludes the proof of Part (a).
\par
The proof of Part (b) relies on the analyticity of the eigenvalues $\tau_{n,\ell}(\l)$ with respect to
$\l$ for each pair of integers $n\geq 0$ and $\ell\geq 0$. This is obvious for $n=0$. So, suppose $n\geq 1$ and note that the potential energy introduced in  \eqref{2.14},
$$
F(\l,w)=\tfrac{\l}{2}w^2-\tfrac{b\mu}{d}\left[w-\ln(1+w)\right],\qquad w>-1,
$$
is analytic with respect to $\l$ and $w>-1$. Thus, also is analytic the function
$$
    G(\l,w):=F(\l,w)-F(\l,w_{n,\pm}),
$$
where $0<w_{n,-}<w_0<w_{n,+}$ are the unique values of $w$ satisfying
$$
   \mathscr{E}(w_n,z)=F(\l,w_{n,\pm}).
$$
Moreover, since $ G(\l,w_{n,\pm})=0$ and
$$
  \frac{\p G}{\p w}(\l,w_{n,\pm})=\frac{\p F}{\p w}(\l,w_{n,\pm})=\l w_{n,\pm}-\frac{b\mu}{d}
  \frac{w_{n,\pm}}{1+w_{n,\pm}}\neq 0,
$$
by the Implicit Function Theorem, $w_{n,\pm}=w_{n,\pm}(\l)$ is analytic with respect to $\l$. Therefore, since either $w_{n}(0)=w_{n,-}(\l)$, or $w_{n}(0)=w_{n,+}(\l)$, are the initial values of the underlying  Cauchy problem, by Peano's Differentiability Theorem, we can infer that the $(n,w_0)$-nodal solution $w_{n}:=w_{n}(\cdot,\l)$ also is an analytic function of $\l$ in $\mc{C}_N^2[0,1]$. Therefore, by a classical result on perturbation from simple eigenvalues of Kato \cite{Ka60}, it follows from \eqref{3.25} that, for every integer $n\geq 1$ and any $(n,w_0)$-nodal solution, $w_n$, of \eqref{2.2}, $\tau_{n,\ell}(\l)$ also is an analytic function of $\l \in(\l_{n}^-,\l_{n}^+)$.
\par
According to Part (a), we already know that
$$
   \tau_{n,n-1}(\l)<0<\tau_{n,n}(\l)
$$
for $\l>\l_{n}^-$ sufficiently close, and   $\l<\l_{n}^+$ sufficiently close. In particular,
\eqref{3.28} holds for these $\l$'s. By the Identity Principle, $\tau_{n,n-1}(\l)$ and
$\tau_{n,n}(\l)$ vanish, at most, at finitely many points in the interval $(\l_{n}^-,\l_{n}^+)$.
Moreover, \eqref{3.28} must be respected for all $\l \in (\l_{n}^-,\l_{n}^+)$, as, otherwise, either
$\tau_{n,n-1}(\l)$, or $\tau_{n,n}(\l)$, must change sign in $(\l_{n}^-,\l_{n}^+)$, and this is impossible. Indeed, suppose that, e.g., $\tau_{n,n-1}(\l)$ changes sign at some $\l_0\in (\l_{n}^-,\l_{n}^+)$. Then, by analyticity, there exists an odd $\nu=2m+1 \geq 1$ such that
$$
   \frac{{\rm d}^\nu \tau_{n,n-1}}{{\rm d}\l^\nu}(\l_0)\neq 0.
$$
Thus, according to Theorem 4.4.3 of \cite{LGMC07}, the generalized algebraic multiplicity
of \cite{LG01} (going back to Esquinas and L\'{o}pez-G\'{o}mez \cite{ELG87} for transversal eigenvalues), satisfies
$$
\chi[\ms{L}_n(\l);\l_0]=2m+1\in 2\mathbb{N}+1,
$$
where
$$
   \ms{L}_n(\l):=-D^2-\l+\tfrac{b\mu}{d}\tfrac{1}{(1+w_{n}(\l))^2}.
$$
Therefore, by  Theorem 4.3.4 of \cite{LG01},  a bifurcation to non-constant positive solutions of \eqref{1.iv} from $(\l,w)=(\l_0,w_{n})$ occurs, which contradicts the uniqueness established by Theorem \ref{th1.1} and ends the proof of Part (b). Finally, Part (c) holds by analyticity.
\end{proof}

\setcounter{equation}{0}
\section{Multiplicity of coexistence states for \eqref{1.iii}}
\label{sec4}

\noindent This section analyzes the existence of coexistence states for the predator-prey model
\begin{equation}
\label{4.1}
\left\{
\begin{array}{ll}
-w''=\lambda w - \varepsilon a(x)w^2 - b\frac{wv}{1+w}\qquad \hbox{in}\;\;(0,1),\\[6pt]
-v''=\mu v -dv^2+ \varepsilon c(x)\frac{wv}{1+w}\qquad\,\,\,\,\hbox{in}\;\;(0,1),\\[7pt]
w'(0)=w'(1)=0,\;\; v'(0)=v'(1)=0
\end{array}
\right.	
\end{equation}
for sufficiently small $\e\geq 0$. These solutions  will be constructed by perturbation from
the coexistence states of the limiting system
\begin{equation}
\label{4.2}
\left\{
\begin{array}{ll}
-w''=\lambda w - b\frac{wv}{1+w}\qquad \hbox{in}\;\;(0,1),\\[6pt]
-v''=\mu v -dv^2 \qquad\quad\;\hbox{in}\;\;(0,1),\\[7pt]
w'(0)=w'(1)=0,\;\; v'(0)=v'(1)=0,
\end{array}
\right.	
\end{equation}
which is \eqref{4.1} with $\e=0$. Subsequently, we set $\mathbb{R}^+:=[0,+\infty)$ and consider the differential operator
\begin{equation*}
\ms{G}:\mathbb{R}^+\times\mathbb{R}^+\times\mathbb{R}^+\times \mc{C}_{N}^2[0,1]\times\mc{C}_{N}^2[0,1]\longrightarrow \mc{C}[0,1]\times\mc{C}[0,1],
\end{equation*}
defined, for every $\varepsilon$, $\l$, $\mu\in\mathbb{R}^+$ and $w,v\in \mc{C}_{N}^2[0,1]$, by
\begin{equation}
\label{4.3}
\ms{G}(\varepsilon,\l,\mu,w,v):=
\begin{pmatrix}
-w''-\l w+\varepsilon a(x)w^2+b\frac{wv}{1+w}\\[6pt]
-v''-\mu v+dv^2-\varepsilon c(x)\frac{wv}{1+w}
\end{pmatrix},
\end{equation}
whose zeroes are the solutions of \eqref{4.1}; the  zeroes of $\ms{G}$ with $\varepsilon=0$ are the solutions of \eqref{4.2}. In Section \ref{sec2}, we have constructed the set of non-constant coexistence states of \eqref{4.2}, which are given by
\begin{equation}
\label{4.4}
(\e,\l,\mu,w,v)=(0,\l,\mu,w_{n},\tfrac{\mu}{d}),\qquad \l\in(\l_{n}^-,\l_{n}^+)
\end{equation}
for every integer $n\in \{1,...,\k\}$ if $\mu \in (\mu_\k,\mu_{\k+1}]$, where $w_{n}$ is any of the two positive $(n,w_0)$-nodal solutions of the first equation of \eqref{4.2} for $v=\tfrac{\mu}{d}$. We already know that they  bifurcate supercritically at $\l=\l_{n}^-$ and subcritically at $\l=\l_{n}^+$ from the constant coexistence state $(w,v)=(w_0,\frac{\mu}{d})$.
\par
A given coexistence state $(w_n,\frac{\mu}{d})$, $n\geq 0$, of \eqref{4.2} is said to be non-degenerate if the linearized operator $D_{(w,v)}\ms{G}(0,\l,\mu,w_{n},\tfrac{\mu}{d})$ cannot admit a zero eigenvalue, i.e., if the eigenvalue problem
\begin{equation}
\label{4.5}
\left\{
\begin{array}{ll}
\begin{pmatrix}
-D^2-\l+\frac{b\mu}{d}\frac{1}{(1+w_n)^2} & b\frac{w_n}{1+w_n}\\[6pt]
0 & -D^2+\mu
\end{pmatrix}
\begin{pmatrix}
\varphi\\[6pt]
\psi
\end{pmatrix}
=\tau
\begin{pmatrix}
\varphi\\[6pt]
\psi
\end{pmatrix},
\\[15pt]
\varphi'(0)=\varphi'(1)=0,\quad \psi'(0)=\psi'(1)=0,
\end{array}
\right.
\end{equation}
does not admit any zero eigenvalue. Suppose that $\tau=0$ is an eigenvalue of \eqref{4.5} for some
$\l\in(0,\frac{b\mu}{d})$ and $\mu>0$.  Then,
$$
  -\psi''+\mu\psi=0
$$
and, since $\mu>0$, necessarily $\psi=0$. Consequently, $\v$ must solve
\begin{equation*}
\left\{
\begin{array}{ll}
\big(-D^2-\l+\frac{b\mu}{d}\frac{1}{(1+w_n)^2}\big)\varphi=0\quad \hbox{in}\;\;(0,1),\\[5pt]
\varphi'(0)=\varphi'(1)=0.
\end{array}
\right.
\end{equation*}
Therefore, for some integer $\ell\geq 0$,
\begin{equation}
\label{4.6}
  \tau_{n,\ell}(\l)=0.
\end{equation}
Therefore,  the coexistence state $(w_n,\frac{\mu}{d})$ is non-degenerate if $\tau_{n,\ell}(\l)\neq 0$ for all $\ell\geq 0$. Consequently, according to Theorem \ref{th3.2}, for every $\mu >\mu_\k$ and
$n\in\{1,...,\k\}$, there is a finite set of singular values of $\l$,
$$
\ms{S}_n \subset (\l_{n}^-,\l_{n}^+),
$$
possibly empty, such that, for every $\l\in(\l_{n}^-,\l_{n}^+)\setminus \ms{S}_n$,
the coexistence state $(\l,\mu,w_n,\tfrac{\mu}{d})$ of \eqref{4.2} is non-degenerate. Moreover,
by the analysis of Section \ref{sec2}, the constant coexistence state
$$
(\l,\mu,w_0,\tfrac{\mu}{d}),\qquad\l\in(0,\tfrac{b\mu}{d})
$$
of \eqref{4.2} is non-degenerate if and only if $\l\neq\l_{n}^{\pm}$ for all integer $n\geq 1$. Thus, setting
\begin{equation*}
    \mathscr{A}:=\big[0,\tfrac{b\mu}{d}\big]\setminus\bigcup_{n=1}^\k(\ms{S}_n\cup\{\l_{n}^{\pm}\}),
\end{equation*}
we find that, for every $\l\in\ms{A}$, any coexistence state of \eqref{4.2} is non-degenerate if
$\mu\in (\mu_\k,\mu_{\k+1})$. In such a case, we will denote
\[
\mathscr{A}_{n}:=
\begin{cases}
\left[\mathscr{A}\cap(\l_{n}^-,\l_{n}^+)\right]
\setminus (\l_{n+1}^-,\l_{n+1}^+)\quad&\mbox{if}\quad n\in\{1,\dots,\k-1\},\\
\mathscr{A}\cap(\l_{\k}^-,\l_{\k}^+)\quad&\mbox{if}\quad n=\k.
\end{cases}
\]
The next result establishes the existence of an arbitrarily large number of
coexistence states for the problem \eqref{4.1} for sufficiently large $\mu$ and sufficiently small $\e>0$.

\begin{theorem}
\label{th4.1}
Suppose that $\mu\in(\mu_\k,\mu_{\k+1})$ for some integer $\k\geq1$, and $n\in\{1,\dots,\k\}$. Then,
for every $\l\in\mathscr{A}$, there exists $\varepsilon_0=\e_0(\l)>0$ such that, for each $\varepsilon\in(0,\varepsilon_0)$, the problem \eqref{4.1} possesses, at least, one coexistence state.
Moreover,  $\e_0$ can be chosen sufficiently small so that, for every $\l\in\ms{A}_n$ and
$\e\in (0,\e_0)$, \eqref{4.1} can admit, at least, $2n+2$ coexistence states. Moreover,
as a consequence of the Implicit Function Theorem, $2n+1$ among these coexistence states perturb, as $\e>0$ perturbs from $0$, from each of the coexistence states
$$
   (w_0,\tfrac{\mu}{d}),\quad (w_n,\tfrac{\mu}{d}),\quad
    (\tilde w_n,\tfrac{\mu}{d}),\quad n\in \{1,...,\k\},
$$
of the limit problem \eqref{4.2},  where, recalling \eqref{2.21}, we are denoting,
$$
  \tilde w_n(x):=\hat w_n(x+n^{-1}),\qquad x\in (0,1).
$$
\end{theorem}

\begin{proof}
Since $\l\in \ms{A}$, $(w_n,\tfrac{\mu}{d})$ is a non-degenerate coexistence state of \eqref{4.2}, whatever be the positive $(n,w_0)$-nodal solution $w_n$ of \eqref{1.iv}. The Fr\'{e}chet differential of $\ms{G}$ at
$(w_n,\tfrac{\mu}{d})$ is given by the linear operator
$$
    \ms{D}:=D_{(w,v)}\ms{G}(0,\l,\mu,w_n,\tfrac{\mu}{d}):\;\mc{C}^2_{N}[0,1]\times
    \mc{C}^2_{N}[0,1] \longrightarrow \mc{C}[0,1]\times \mc{C}[0,1],
$$
defined, for every $u$, $v\in\mc{C}^2_{N}[0,1]$,  by
\begin{equation*}
\ms{D}(w,v) :=\begin{pmatrix}
-w''-\l w+\frac{b\mu}{d}\frac{w}{(1+w_n)^2}+b\frac{w_n v}{1+w_n}\\[4pt]
-v''+\mu v
\end{pmatrix}.
\end{equation*}
Since $(w_n,\tfrac{\mu}{d})$ is non-degenerate,
$$
   N[\ms{D}]=\mathrm{span\,}[(0,0)],
$$
i.e., $\ms{D}$ is injective. As $\ms{D}$ is a Fredholm operator of index zero, by the Open Mapping Theorem, $\ms{D}$ is a topological isomorphism. Therefore, by the Implicit Function Theorem, there exist $\varepsilon_0>0$
and two analytic functions
$$
  W_n, V_n : (-\e_0,\e_0)\times (\l-\e_0,\l+\e_0)\times (\mu-\e_0,\mu+\e_0) \to \mc{C}^2_{N}[0,1]\times  \mc{C}^2_{N}[0,1]
$$
such that, for some $\varphi_n$, $\psi_n\in\mc{C}_{N}^2[0,1]$,
\begin{equation}
\label{4.7}
W_{n}(\e,\l,\mu)=w_n+\e \varphi_n+O(\varepsilon^2),\qquad V_n(\varepsilon,\l,\mu)=\tfrac{\mu}{d}+\e\psi_n +O(\varepsilon^2),
\end{equation}
and
$$
\ms{G}(\varepsilon,\l,\mu,W_{n}(\e,\l,\mu),V_n(\varepsilon,\l,\mu))=0
$$
for all
$$
   (\e,\l,\mu)\in (-\e_0,\e_0)\times (\l-\e_0,\l+\e_0)\times (\mu-\e_0,\mu+\e_0).
$$
Moreover, these are the unique solutions of \eqref{4.1} in a neighborhood of
$$
  (\e,\l,\mu,w,v)=(0,\l,\mu,w_n,\tfrac{\mu}{d}).
$$
Furthermore, since $w_n-w_0$ has $n$ simple zeroes in $(0,1)$, also $W_n-w_0$ has $n$ simple zeroes in
$(0,1)$.
\par
Summarizing, the Implicit Function theorem provides us with a total of $2n+1$ coexistence states of
\eqref{4.1} for sufficiently small $\e>0$ perturbing from the $2n+1$ non-degenerate coexistence states
of \eqref{4.2}: one of them constant, $(w_0,\frac{\mu}{d})$, $n$ of type $(w_j,\frac{\mu}{d})$, and $n$ of type $(\tilde w_j,\frac{\mu}{d})$, where $w_j$ and $\tilde w_j$ stand for the two positive $(j,w_0)$-nodal  solutions of \eqref{1.iv}, $j\in\{1,...,n\}$.
\par
Finally, note that, thanks to Lemma 5.1 and Theorem 5.1 of \cite{LGMH24}, the coexistence states
of \eqref{4.1} have a priori bounds for all $\e>0$. Thus, the component of the solution set of
\eqref{4.1} containing the local perturbation of the curve of
constant coexistence states $(\l,w_0,\frac{\mu}{d})$, $\ms{C}_0$, denoted by $\ms{C}_\e$, must contain at least two constant coexistence states for each $\l\in\ms{A}$.  This ends the proof.
\end{proof}

Next, we will determine the functions $\v_n$ and $\psi_n$ arising in the asymptotic expansions \eqref{4.7}.
Substituting $(W_n(\varepsilon,\l,\mu),V_n(\varepsilon,\l,\mu))$ in \eqref{4.1}, taking into account that
$(w_n,\frac{\mu}{d})$ solves \eqref{4.2}, dividing by $\e$ and letting $\e\to 0$, it becomes apparent that
$$
\left\{
\begin{array}{ll}
\Big(-D^2-\l+\tfrac{b\mu}{d}\tfrac{1}{(1+w_n)^2}\Big)\varphi_n=-a(x)w_n^2-b\tfrac{w_n}{1+w_n}\psi_n\varphi\quad&\hbox{in}\;\;(0,1),\\[7pt]
(-D^2+\mu)\psi_n=\tfrac{\mu}{d}\tfrac{c(x) w_n}{1+w_n}\varphi\quad&\hbox{in}\;\;(0,1),\\[7pt]
\varphi'_n(0)=\varphi'_n(1)=0,\quad
\psi'_n(0)=\psi'_n(1)=0.
\end{array}
\right.
$$
Since the lowest eigenvalue satisfies
$$
   \tau_0(\mu)=\mu>0,
$$
by applying Theorem 7.10 of \cite{LG13}, it follows from the $\psi_n$-equation that
$$
\psi_n=\tfrac{\mu}{d}(-D^2+\mu)^{-1}\Big(\tfrac{c(x)w_n}{1+w_n}\Big)\gg 0.
$$
Thus, $V_n(\varepsilon,\l,\mu)$ is inhomogeneous for all $n\geq1$.
\par
Lastly, since $\l\in \ms{A}$, Theorem \ref{th3.2} entails that
$$
     \tau_{n,\ell}(\l)=\tau_\ell\left(-\l+\tfrac{b\mu}{d}\tfrac{1}{(1+w_n)^2}\right)\neq 0\quad
\hbox{for all integer}\;\;  \ell\geq 0.
$$
Hence,  it follows from the $\v_n$-equation that
$$
\varphi_n=\Big(-D^2-\l+\tfrac{b\mu}{d}\tfrac{1}{(1+w_n)^2}\Big)^{-1}
\Big(-a(x)w_n^2-b\tfrac{w_n}{1+w_n}\psi_n\Big).
$$
\par
Adopting the notations introduced in Section \ref{sec2}, the Implicit Function Theorem can be used to show
that the compact arcs of the closed loops $\ms{C}_n$ whose
$\l$-projections are included in $\ms{A}$ provide us with compact arcs of curve filled in by coexistence states of \eqref{4.1}. However, this does not guarantee that the global structure of the solution set of the uncoupled system \eqref{4.2} will be maintained as $\e$ perturbs from zero. A sharper analysis of this particular issue will be carried out elsewhere.

\end{document}